\documentclass[12pt, oneside]{article}   	

\usepackage[english]{babel}
 
\usepackage{natbib}
\setcitestyle{authoryear,open={(},close={)}}
\usepackage{geometry} 
\usepackage{amsmath}
\usepackage{amssymb,amsfonts}        
\usepackage{newpxtext} 
\usepackage{bm}      		
\usepackage{graphicx}				
\usepackage{bbm}								
\usepackage{amssymb}
\usepackage{amsthm}
\newtheorem{theorem}{Theorem}[section]
\newtheorem{lemma}[theorem]{Lemma}

\newtheorem{definition}[theorem]{Definition}
\newtheorem{remark}[theorem]{Remark}
\usepackage{breqn}
\usepackage{authblk}
\usepackage{tikz}
\usetikzlibrary{decorations.pathreplacing}
\usetikzlibrary{arrows,positioning}

\def\shape#1{
  \lower5pt\hbox{
  \hskip-7pt
  \tikzset{circ/.style={circle, draw, fill=black, scale=.2}}
  \begin{tikzpicture}[semithick,scale=.3]
  \node (l1) at (0,.866) [circ]{};
  \node (l2) at (1,.866) [circ]{};
  \node (l3) at (0.5,0) [circ]{};
  #1
  \end{tikzpicture}
  \hskip-8pt}
}
\def\Vshape{
  \draw[-,color=white] (l1) to node [auto] {} (l2);
  \draw[-] (l1) to node [auto] {} (l3);
  \draw[-] (l2) to node [auto] {} (l3);
}

\begin{document}

\title{Optimal Single Sample Tests for Structured \\ versus Unstructured Network Data}


\author[1]{Guy Bresler\thanks{guy@mit.edu}}
\author[1]{Dheeraj Nagaraj\thanks{dheeraj@mit.edu}}

\affil[1]{ Massachusetts Institute of Technology}


\maketitle

\begin{abstract}%
 We study the problem of testing, using only a single sample, between mean field distributions (like Curie-Weiss, Erd\H{o}s-R\'enyi) and structured Gibbs distributions (like Ising model on sparse graphs and Exponential Random Graphs). 
Our goal is to test without knowing the parameter values of the underlying models: only the \emph{structure} of dependencies is known.  
We develop a new approach that applies to both the Ising and Exponential Random Graph settings based on a general and natural statistical test. The test can distinguish the hypotheses with high probability above a certain threshold in the (inverse) temperature parameter, and is optimal in that below the threshold no test can distinguish the hypotheses.

 The thresholds do not correspond to the presence of long-range order in the models. By aggregating information at a global scale, our test works even at very high temperatures.
 The proofs are based on distributional approximation and sharp concentration of quadratic forms, when restricted to Hamming spheres. The restriction to Hamming spheres is necessary, since otherwise any scalar statistic is useless without explicit knowledge of the temperature parameter.  At the same time, this restriction radically changes the behavior of the functions under consideration, resulting in a much smaller variance than in the independent setting; this makes it hard to directly apply standard methods (i.e., Stein's method) for concentration of weakly dependent variables. Instead, we carry out an additional tensorization argument using a Markov chain that respects the symmetry of the Hamming sphere. 
%
%
%
%
%
\end{abstract}

\section{Introduction}

Hypothesis testing for network data has received a lot of attention in  recent years. There are two basic types of network data: first, the network or graph itself; and second, observations from the nodes in a network, where the network describes interactions between the nodes.
A recent example of the first type is studied in the paper of \cite{bubeck2016testing}, which gives an optimal single-sample test to distinguish between geometric random graphs and Erd\H{o}s-R\'enyi random graphs by counting the number triangles in the graph. Similarly, \cite{gao2017testing} use distributional approximation for a specific statistic to distinguish between an Erd\H{o}s-R\'enyi random graph and sample from the Stochastic Block model. Another paper in this direction is that of \cite{ghoshdastidar2017two}, who consider the problem of deciding whether two given graphs are samples from the same graph model or from two different models. Their method is based on existence of a statistic that concentrates at different values for the two different graph models. The problem of testing if a known graph (with atleast $\Omega(\log{n})$ vertices) is planted in a sample from Erd\H{o}s-R\'enyi random graph with known edge parameter was studied by \cite{javadi2015statistical}. They give sharp single sample thresholds for the problem and the corresponding statistical test which can achieve this threshold. As will be seen below, our result on testing graph model differs in the fact that we consider the appearance of much smaller subgraphs and the subgraphs are not `planted' explicitly. 

 As far as data from nodes in a network is concerned, \cite{martin2017exact} considers the problem of tractably finding goodness-of-fit for Ising models.  \cite{daskalakis2016testing} developed methods for testing whether samples are coming from a given known Ising model. They assume full knowledge of all the model parameters, use a test based on the empirical estimation for pairwise correlations among sites. Their sample complexity guarantees are polynomial in $n$, whereas use the special structure present in the mean-field case to give sharp threshold above which single sample testing is possible using a general framework applicable to other models.
 \cite{daskalakis2017concentration} 
and \cite{gheissari2017concentration} show concentration for polynomials of Ising models at high temperature, and improve the sample complexities obtained in \cite{daskalakis2016testing} for testing whether samples are from the product distribution (i.e., coordinates are independent) or from an Ising model $\pi$ guaranteed to have KL-divergence at least $\epsilon$ from the product distribution. Analogously, \cite{canonne2017testing} consider the problem of determining whether observed samples from a distribution $P$ agree with a known fully-specified Bayesian network $Q$, using multiple samples; and also the problem of testing whether two unknown Bayes nets are identical or not, using multiple samples. Finally, this latter paper considers also \emph{structure testing}, i.e., testing if samples are from a Bayes net with a certain structure. Our objectives differ from these papers in that: 1) our test is based on a single sample; and 2) there are no assumptions of separation in KL-divergence or total variation on the distributions generating the sample. Instead, the guarantees are in terms of the natural model parameter. \cite{mukherjee2013consistent} considers the problem of consistent parameter estimation of the two star (wedge graph) ERGM considered in this paper. Their method assumes that the strength of the 'wedge interaction' $\beta_2 \in (0,\infty)$ is fixed. Whereas, in our work, the sharp threshold for distinguishing this graph from Erd\H{o}s-R\'enyi graphs is shown to be $\beta_2 = \Theta(\frac{1}{\sqrt{n}})$, which goes to $0$ with $n$. It is unclear how their parameter estimation methods can be used in this case to obtain the sharp threshold behavior.

In this paper we prove an abstract result, Theorem~\ref{main_theorem}, that provides a framework for establishing near-optimal hypothesis tests between data from a network with a given dependency structure (like Ising model, Exponential Random Graph Model) and unstructured data (like Curie-Weiss, Erd\H{o}s-R\'enyi). We do not assume knowledge of model parameters, which makes the problem more challenging, but also more applicable to many settings where there is no way to learn them accurately based on one sample. 

As the first of two applications developed in this paper, we consider the problem of testing whether the network data comes from an Ising model over a known $d$-regular graph with unknown inverse temperature $\beta$ (also with possibly nonzero external field) or alternatively from a permutation invariant distribution (which includes the Curie-Weiss model at unknown temperature). In Section~\ref{sec:ising_model_result} we note that Curie-Weiss model is indeed the worst adversary. We motivate this problem by discussing an adversarial data scenario in Section~\ref{subsec:adversary_model}. 

\begin{theorem}[Informal version of Theorem~\ref{thm:ising_threshold}]
We can distinguish Ising models on $d$-regular graphs from the Curie-Weiss model (complete graph) with high probability with one sample if the inverse temperature $\beta$ of the Ising model satisfies $\beta\sqrt{nd} \to \infty $. Conversely, if $\beta\sqrt{nd} \to 0$, then there is no statistical test that can distinguish them with high probability, even using a constant number of i.i.d. samples.
\end{theorem} 

\begin{remark}
We interpret the result above as follows: whenever $\beta\sqrt{nd} \to \infty$, an adversary cannot come up with a Curie-Weiss sample at some temperature such that it can be confused for a sample from the $d$ regular Ising model. Conversely, whenever $\beta\sqrt{nd} \to 0$, the adversary can choose a Curie-Weiss model at a specific temperature depending only on $\beta$ such that the total variation distance between these distributions converges to $0$. The problem is formulated in the minimax sense. 
\label{rem:explain_ising_result}
\end{remark}

The result works for every $d$-regular graph. It was shown in \cite{bresler2017stein} that pairwise correlations, and more generally $k$th-order moments, of the Curie-Weiss model can be well approximated on average by expander graphs, yet the result above holds even when the underlying graph is an expander. The test also works deep inside the high temperature regime ($\beta \leq \Theta(\frac{1}{d})$), when there is no global order, by aggregating small dependencies from the entire network.

 Our results also apply to certain random graph distributions, and
in Section~\ref{sec:ergm_result} we apply our framework to compare $ G(n,p_n)$ (the Erd\H{o}s-R\'enyi model) and exponential random graphs. 
 Let $\mathrm{ERGM}(\beta_1,\beta_2)$ be the exponential random graph with respect to the single edge $E$ and the $V$-graph \hbox{(\shape{\Vshape})} with inverse temperature parameters $\beta=(\beta_1,\beta_2) \in \mathbb{R}^2$. The parameter $\beta_1$ controls edge density, while $\beta_2$ encourages presence of $V$-subgraphs.
 
\begin{theorem}[Informal version of Theorem~\ref{thm:ergm_threshold}]
We can distinguish $G(n,p)$ and $\mathrm{ERGM}(\beta)$ with high probability with one sample if $\sqrt{n}\beta_{2}  \to \infty $. Conversely, if $\sqrt{n}\beta_{2} \to 0$, then there is no statistical test which can distinguish them with high probability using constant number of i.i.d. samples.
\end{theorem}

\begin{remark}
Similar to the result on Ising models, we interpret this result as follows: whenever $\beta_2\sqrt{n} \to \infty$, we can distinguish the $\mathrm{ERGM}(\beta_1,\beta_2)$ from $G(n,p)$ for any unknown $p$ and $\beta_1$ (under certain constraints on values taken by $p$). Conversely, whenever $\beta_2\sqrt{n} \to 0$, we can choose $p$ and $\beta_1$ such that the total variation distance between these distributions converges to $0$. Specifically, we can distinguish between these models even when the edge density in these models are the same as long as $\beta_2\sqrt{n} \to \infty$.
\end{remark}
In~\cite{bhamidi2011mixing} it is shown that in the high-temperature regime $\beta_2 \leq \Theta(1)$, any finite collection of $k$ edges converges in distribution to independence. (In $G(n,p)$ all edges are independent.) Our test aggregates global information to distinguish between them and works when the dependence parameter $\beta_2$ is much smaller than the high-temperature threshold. ~\cite{bhamidi2011mixing} and \cite{eldan2017exponential} consider existence of unique solutions to a certain fixed point equation to define the high temperature regime in ERGMs. We use an entirely different method to identify the phases in our setup -- where we choose parameters of degree 2 polynomials of binomial random variables to minimize the variance -- to choose $\beta_1$ as a function of $\beta_2$ and $p$ such that $\mathsf{ERGM}(\beta_1,\beta_2)$ converges in total variation distance to $G(n,p)$ whenever $\beta_2 \sqrt{n} \to 0$. This is illustrated in Appendix~\ref{sec:super_concentration}.

\paragraph{Outline.} The next subsection motivates our results with an adversarial data detection scenario. Section~\ref{s:notation} introduces notation and defines the Ising and exponential random graph models, formulates the exact statistical problem as well as gives intuition for the statistical test we use in our applications.
In Section~\ref{sec:abstract_result} we state our abstract hypothesis testing result, which is based on distributional approximation.
In Section~\ref{sec:ising_model_result} we apply our framework to prove Theorem~\ref{thm:ising_threshold} for the Ising model. 
In Section~\ref{sec:clt} we prove the required distributional approximation for quadratic forms using Stein's method and in Section~\ref{sec:burkholder_subexp} we prove sharp concentration inequalities for quadratic forms over the Hamming sphere using a novel method.

\subsection{Motivating example: detecting fraudulent data}
\label{subsec:adversary_model}
Suppose that we have collected responses to a survey from a set of people, indicating a binary preference for something (iPhone or Android, Democrat or Republican, etc.). Moreover, we have access to the network structure $G_0$ (e.g., induced Facebook subgraph) and the data is modeled by a family of probability distributions $\{Q_{G_0,\lambda}: \lambda \in \Lambda\}$ (e.g., Ising models on $G_0$) for some parameter set $\Lambda$.  An adversary may attempt to counterfeit the data generated by the network using instead a distribution $P$, possibly biased (e.g., to fix an election). We assume that the adversary may know the graph, but does not know the labeling of the nodes. The adversary therefore seeks to minimize the probability of the tampering being detected, which amounts to minimizing $\mathbb{E}_\pi  \inf_{\lambda\in \Lambda} d_{\mathsf{TV}}(P, Q_{\pi G,\lambda})$, where $\pi$ is a uniformly random permutation encoding the adversary's prior over the node labels.

The analysis of the quantity $\mathbb{E}_\pi  \inf_{\lambda\in \Lambda} d_{\mathsf{TV}}(P, Q_{\pi G,\lambda})$ is fairly involved and requires convexity of the class of distributions. Our framework is able to handle testing against a convex combination of distributions, but for this manuscript we instead relax this objective to $\inf_{\lambda \in \Lambda} \mathbb{E}_{\pi} d_{\mathsf{TV}}(P, Q_{G,\lambda})$.
%

For any permutation $\pi$, let the distribution $\pi P$ be defined by $\pi P(x) = P(\pi(x))$. For arbitrary $\lambda \in \Lambda$,
\begin{align}
 \mathbb{E}_\pi d_{\mathsf{TV}}(P,Q_{\pi G,\lambda}) &=\mathbb{E}_{\pi} d_{\mathsf{TV}}(\pi^{-1}P,Q_{G,\lambda}) \nonumber \\
&=  \frac{1}{n!} \sum_{\pi}  d_{\mathsf{TV}}(\pi^{-1}P,Q_{G,\lambda}) \nonumber \\
&\geq  d_{\mathsf{TV}}\left(\sum_{\pi}\tfrac{1}{n!}\pi^{-1}P, Q_{G,\lambda}\right) \,.\label{eq:expectation_minimization}
\end{align}
In the first step, we have used the fact that $\pi Q_{G,\lambda} \stackrel{d}{=} Q_{\pi G,\lambda}$ (due to relabeling of vertices).  In the third step we have used Jensen's inequality for the convex function $d_{\mathsf{TV}}$.
Clearly, the distribution $\hat{P} := \frac{1}{n!}\sum_{\pi}\pi^{-1}P$ is permutation invariant. 

If there is a unique optimal distribution $P_0$ for the adversary, we conclude that it must be a permutation invariant distribution. Some of the key features of the problem above are:  There is only one sample available, the underlying network structure and model is known and the adversary, who is agnostic to the network structure, comes up with permutation invariant data to mimic the data from the network. The considerations above justify our hypothesis testing model in Section~\ref{sec:ising_model_result} where the true network data is taken to be from an Ising model.

\section{Notation and Definitions}\label{s:notation}
$\mathbb{E}_p f$ denotes the expectation with respect to the probability measure $p$. For any two probability measures $\mu$ and $\nu$ over $\mathbb{R}$, we denote the Kolmogorov-Smirnoff distance as $d_{\mathsf{KS}}(\mu,\nu) := \sup_{x_0 \in \mathbb{R}}|\mu(\{x:x \leq x_0\} ) - \nu(\{x: x \leq x_0\})|\,.$
Let $\mathsf{Lip}_1(\mathbb{R})$ be the class of all $1$-Lipschitz real-valued functions over $\mathbb{R}$. For $\mu$ and $\nu$ probability measures over $\mathbb{R}$, the Wasserstein distance is defined as:
$d_{\mathsf{W}}(\mu,\nu) = \sup_{f \in \mathsf{Lip}_1(\mathbb{R})} \mathbb{E}_{\mu}f -\mathbb{E}_{\nu}f$.
For any random variable $X$, let $\mathcal{L}(X)$ be the probability law of $X$. Let $\Phi(x)$ denote the standard normal cumulative distribution function.

\subsection{Ising Model}
The \emph{interaction matrix} $J$ is a real-valued symmetric $n \times n$ matrix with zeros on the diagonal and the \emph{external field} is a real number $h$. Define the Hamiltonian $\mathcal{H}_{J,h}: \{-1,1\}^n \to \mathbb{R}$ by $\mathcal{H}_{J,h}(x) = \frac{1}{2}x^{\intercal}Jx + h\left(\sum_i x_i\right)$. Construct the graph $G_J = ([n],E_J)$ with $(i,j) \in E_J$ iff $J_{ij} \neq 0$.  An Ising model over graph $G_J$ with interaction matrix $J$ and external field $h$ is the probability measure $\pi$ over $\{-1,1\}^n$ such that $\pi(x) \propto \exp{(H_J(x))}$.

For any simple graph $G = ([n],E)$ there is an associated symmetric $n\times n$ adjacency matrix $\mathcal{A}(G) := (\mathcal{A}_{ij})$, where
$\mathcal{A}_{ij} = 
1$ if $(i,j) \in E$ and 
$\mathcal{A}_{ij} = 0$ otherwise.

 Let $K_n$ be the complete graph on $n$ nodes.
The \emph{Curie-Weiss model} at inverse temperature $\beta^{\mathrm{CW}} > 0$ and external field $h^{\mathrm{CW}}$ is the Ising model with interaction matrix $\frac{\beta^{\mathrm{CW}}}{n}\mathcal{A}(K_n)$. It can be easily shown this corresponds to the distribution $p(x) \propto e^{\frac{\beta^{\mathrm{CW}}}{2}nm^2 + nh^{\mathrm{CW}}m}$, where $m = m(x) = \frac{1}{n}\sum_{i=1}^n x_i$ is called the magnetization. 
The Curie-Weiss model is an unstructured/mean field model. It is permutation invariant and assigns the same probability to states with the same magnetization $m(x)$.

We will compare the above model to the Ising model on a $d$-regular graph $G_d = ([n],E_d)$ (i.e., every node has degree $d$).  For a given inverse temperature $\beta^{\mathrm{dreg}}$, we consider the Ising model with interaction matrix $\beta^{\mathrm{dreg}}\mathcal{A}(G_d)$ and external field $h^{\mathrm{dreg}}$. We shall call this `$d$-regular Ising model'.

\begin{remark}
The Curie-Weiss model is well studied and it can be shown that it exhibits non-trivial behavior when $\beta^{\mathrm{CW}} = \Theta(1)$. It undergoes a phase transition when $\beta^{\mathrm{CW}} = 1$, below which pairwise correlations are $O(\frac{1}{n})$ and when $\beta^{\mathrm{CW}} > 1$ they are $\Theta(1)$. The pairwise correlations tend to $1$ (maximum possible value) as $\beta^{\mathrm{CW}} \to \infty$. As considered in Section~\ref{sec:ising_model_result}, $\beta^{\mathrm{CW}} \leq \beta_{\mathsf{max}} $ for fixed $\beta_{\mathsf{max}}$ is a natural choice of Curie-Weiss models. 

The non-trivial regime for the d-regular Ising model is $\beta^{\mathrm{dreg}}_n = \theta(\frac{1}{d})$. As shown in Section~\ref{sec:ising_model_result}, our testing works as long as $\beta^{\mathrm{dreg}}_n \gg \frac{1}{\sqrt{nd}}$ which includes the regime of interest. 
\end{remark}

\subsection{Exponential Random Graph Model}
The Erd\H{o}s-R\'enyi random graph model $G(n,p)$ for $p \in [0,1]$ is the distribution of simple graphs on $n$ vertices such that each edge is included independently with probability $p$.

Consider fixed finite simple graphs $\{ H_i \}_{i=1}^{K}$, such that $H_1$ is the graph with two vertices and a single edge. Let $\beta \in \mathbb{R}\times\left(\mathbb{R}^{+}\right)^{K-1}$. Given a graph $G$ over $n$ vertices, define $N_{i}(G)$ to be the number of edge preserving isomorphisms from $H_i$ into $G$ (i.e, no. of subgraphs of $G$ (not necessarily induced), which are isomorphic to $H_i$). In particular $N_1(G)$ is twice the number of edges in $G$. Let $v_i$ be the number of vertices in $H_i$.  In the following definition of Exponential Random Graph Model (ERGM), we follow the convention in \cite{bhamidi2011mixing} to allow the values of $N_i(G)$ to be of the same order of magnitude to allow non-trivial behavior.

We construct the Hamiltonian $\mathcal{H}_{\beta}(G) = \sum_{i=1}^{K}\beta_i \frac{N_i(G)}{n^{v_i -2}}$. Consider the probability distribution $\nu(\cdot)$ over the set of simple graphs over $n$ vertices such that $\nu(G) =\frac{ e^{\mathcal{H}_{\beta}(G) }}{Z(\beta)}$ where $Z(\beta)$ is the normalizing factor. We call the distribution $\nu(\cdot)$ to be the exponential random graph model $\mathrm{ERGM}(\beta)$. We note that when $\beta_i =0$ for $i \geq 2$, $\mathrm{ERGM}(\beta)$ is the same as $G(n,\frac{e^{2\beta_1}}{1+e^{2\beta_1}})$. Rougly speaking, $\mathrm{ERGM}$ is like $G(n,p)$ but it favors the occurrence of certain subgraphs. Therefore, $G(n,p)$ is the mean field model and $\mathrm{ERGM}$ is the structured model.
\begin{remark}
In this paper, we take $K=2$ and fix $H_2$ to be the wedge graph \hbox{(\shape{\Vshape})}.
\end{remark}
\subsection{Problem Formulation}
We formulate our problem as a minimax hypothesis testing problem:

$H_0$: Data is from some mean field model with parameter $\gamma \in \Gamma$. We denote the corresponding distributions over the state space $\Omega$ by $P_{\gamma}$.

$H_1$ : Data is from a structured model with unknown parameters $\lambda \in \Lambda $. We denote the corresponding distributions over the state space $\Omega$ by $Q_{\lambda}$.

We take a statistical test $\mathcal{T}$ to be a decision function $\mathcal{D}_{\mathcal{T}}:\Omega \to \{H_0,H_1\}$. Let $p_1(\gamma, \mathcal{T}):= \mathbb{P}(\mathcal{D}_{\mathcal{T}}(\hat{X}) = H_1| \hat{X} \sim P_{\gamma})$ and $p_2(\lambda,\mathcal{T}) := \mathbb{P}(\mathcal{D}_{\mathcal{T}}(\hat{X}) = H_0| \hat{X} \sim Q_{\lambda})$. We take the risk of the test $\mathcal{T}$ to be worst case Bayesian probability of error:

$$R(\mathcal{T}) = \sup_{\gamma \in \Gamma}\sup_{\lambda \in \Lambda} \max(p_1(\gamma, \mathcal{T}),p_2(\lambda, \mathcal{T}))$$

We elucidate our result with the Ising case over $n$ variables: We fix $$\Gamma = \{(\beta^{\mathrm{CW}},h^{\mathrm{CW}}): 0\leq \beta^{\mathrm{CW}} \leq \beta_{\mathsf{max}}, |h^{\mathrm{CW}}| \leq h_{\mathsf{max}}\} $$ and the corresponding distributions to be Curie-Weiss models at inverse temperature $\beta^{\mathrm{CW}}$ and external field $h^{\mathrm{CW}}$. There are two cases to be considered:
\begin{enumerate}

\item Case 1:
Let $L_n$ be any positive sequence which diverges to infinity.
$$\Lambda_n = \{(\beta^{\mathrm{dreg}}_n,h) : \beta^{\mathrm{dreg}}_n \geq \frac{L_n}{\sqrt{nd}}, |h^{\mathrm{dreg}}| \leq h_{\mathsf{\max}}\}$$
 We show that in this case, for every $n$ there exists a test $\mathcal{T}(L_n,\beta_{\mathsf{max}}, h_{\mathsf{max}})$ such that $R(\mathcal{T}(L_n,\beta_{\mathsf{max}}, h_{\mathsf{max}})) \to 0$.
we explicitly construct this test by considering an elementary hypothesis test in Section~\ref{sec:abstract_result} which compares a specific $P_{\gamma}$ and $Q_{\lambda}$ and extend this to the composite case by proving that the statistical test considered doesn't actually look at the parameters $\gamma$ and $\lambda$. 
\item
Case 2: 
Let $L_n$ be any positive sequence which diverges to infinity.
$$\Lambda_n = \{(\beta^{\mathrm{dreg}}_n,h) : \beta^{\mathrm{dreg}}_n = \frac{1}{L_n\sqrt{nd}}, |h^{\mathrm{dreg}}| \leq h_{\mathsf{max}}\}$$
We show that in this case, for every sequence of tests $\{\mathcal{T}_n\}$, $$ \liminf_{n\to \infty} R(\mathcal{T}_n) \geq \frac{1}{2} $$ - which is the same as random labeling w.p $\frac{1}{2}$. To prove this, we show that for every $\lambda = (\beta^{\mathrm{dreg}}_n,h) \in \Lambda_n$ in this set, we can find $\gamma \in \Gamma$ such that: $$d_{\mathsf{TV}}(P_{\gamma},Q_{\lambda}) \to 0$$
Since $\max(p_1(\gamma, \mathcal{T}),p_2(\lambda, \mathcal{T})) \geq \frac{1}{2}(1-d_{\mathsf{TV}}(P_{\gamma},Q_{\lambda}))$ we conclude the result.
\end{enumerate}

\subsection{Intuition behind the Comparison Result}
\label{subsec: intuition}
Consider Ising model $q(\cdot)$ with interaction matrix $\beta^{\mathrm{dreg}} B$ ($\beta^{\mathrm{dreg}}$ unknown) and the Curie-Weiss model $p(\cdot)$ (at an unknown temperature $\beta^{\mathrm{CW}}$). The measure $q(\cdot)$ assigns higher probability to states with higher value of $x^{\intercal}Bx$, so a natural idea for  distinguishing between $p$ and $q$ would be to check if $x^{\intercal}Bx$ has a large value. However, the inverse temperature parameters are unknown, which implies that we can have the same expected value for the statistic under both the hypotheses (for some choice of temperature parameters). 

Instead, we exploit the symmetry in the Curie-Weiss model to overcome this drawback.
Let $\Omega_n = \{-1,1\}^n$ and consider the magnetization function $m: \Omega_n \to [-1,1]$ given by $m(x) = \frac{\sum_{i=1}^n x_i}{n}$.  Let $A_{m_0} = \{x\in \Omega_n : m(x) = m_0\}$ for $m_0 \in \{-1,-1+\frac{2}{n},\dots,1\} =: M_n$. We can partition $\Omega_n$ as: $\Omega_n = \cup_{m_0 \in M_n} A_{m_0}$.

For the Curie-Weiss model $p(\cdot)$ states with the same magnetization have the same probability. Therefore $p(\cdot)$ gives the uniform distribution over the set $A_{m_0}$. This continues to be case irrespective of the inverse temperature and external field, which mitigates our initial problem. The distribution $q(\cdot)$, given the magnetization $m_0$, assigns most of the probability to states $x$ with large values of $x^{\intercal}Bx$. We first prove a central limit theorem for $g(x):= x^{\intercal}Bx$ when $x$ is drawn uniformly from the set $A_{m_0}$. Then we show that the event $$\frac{g(x) - \mathbb{E}_p [g(x)| x\in A_{m_0}]}{\sqrt{\mathrm{var}_p(g(x)|x\in A_{m_0})}} \geq T$$ has a small probability under $p(\cdot|x\in A_{m_0})$ for large values of $T$, but has a large probability under $q(\cdot)$ because it favors larger values of $g(x)$. This gives us a distinguishing statistic for large enough inverse temperature $\beta^{\mathrm{dreg}}$. 

Similarly, $\mathsf{ERGM}(\beta_1,\beta_2)$ favors the appearance of $V$ subgraphs \hbox{(\shape{\Vshape})} when compared to $G(n,p)$. But, the expected number of \hbox{(\shape{\Vshape})}  subgraphs can be made equal by increasing $p$. To overcome this disadvantage, we exploit the symmetry in $G(n,p)$ - i.e, it assigns the same probability to all graphs with the same number of edges. So, conditioned on the number of edges in the sample graph, we check if the number of \hbox{(\shape{\Vshape})} subgraphs are disproportionately large. We proceed with the exact same framework as the Ising model for this, by proving a central limit theorem for the number of \hbox{(\shape{\Vshape})} subgraphs when the graphs have a constant number of edges.
\section{Abstract Result}
\label{sec:abstract_result}
We consider a sequence of probability spaces $(\Omega_n,\mathcal{F}_n,p_n)$, $n \in \mathbb{N}$. Consider a $\mathcal{F}_n$ measurable, real valued function $g_n$ such that $\mathbb{E}_{p_n} [e^{\beta_n g_n}] < \infty$ for all $\beta_n \in \mathbb{R}$ and define measure $q_n$ using Radon-Nikodym derivative as:
  $$\frac{dq_n}{dp_n} = \frac{e^{\beta g_n}}{\mathbb{E}_{p_n}[e^{\beta_n g_n}]}$$
  
We try to compare the distributions $p_n$ and $q_n$ in the total variation sense. We shall use the notation defined in the following discussion of the abstract result even when dealing with specific examples. Consider the following conditions:
\begin{enumerate}
\item[\textbf{C1}]
For some finite index set $M_n$ such that $|M_n| = M(n) \in \mathbb{N}$, we can partition $\Omega = \cup_{m\in M_n}A_m$ with disjoint sets $A_m$ such that $p_n(A_m) > 0\ \forall \ m \in M_n$.
  \label{partition}
 \item[\textbf{C2}]
 For a set $S_n \subset M_n$,  $$p_n(\cup_{m\in S_n} A_m) \geq 1 - \alpha_n$$ for some sequence $\alpha_n \to 0$.
\label{high_probability}
 \item[\textbf{C3}]
   Let $p^{(m)}$ be the probability measure over $A_m$ defined by $p^{(m)}(A) := \frac{p_n(A)}{p_n(A_m)} \ \forall \ A \subset A_m $ and $A \in \mathcal{F}_n$. It is the projection of the measure $p$ over the set $A_m$. Let $X_m \sim p^{(m)}$ and $X \sim p_n$. Let  $e_m(g_n) := \mathbb{E}[g_n(X_m)] $ and $\sigma^2_m(g_n) := \mathrm{var} \left[{g_n(X_m)}\right]$. For all $m, m^{\prime} \in S_n$, $$ 0< c \leq \frac{\sigma_m^2(g_n)}{\sigma_{m^{\prime}}^2(g_n)} \leq C$$ for some constants $c,C$ independent of $n$. We let $\sigma_n$ be any sequence such that $c\sigma_m(g_n) \leq \sigma_n\leq C\sigma_m(g_n)$  for some absolute constants $c$ and $C$ for every $m \in S_n$. Although $\beta_n$ can be a parameter in $g_n$, $g_n(x)-e_m(g_n)$ doesn't depend on $\beta_n$ whenever $x \in A_m$
   \label{partition_regularity}
  \item[\textbf{C4}] There is a sequence $\tau_n \to 0$ such that
  $$\sup_{m \in S_n} d_{\mathsf{KS}}\left(\mathcal{L}\left(\tfrac{g(X_m) - e_m(g_n)}{\sigma_m}\right),\mathcal{N}(0,1)\right) < \tau_n\,.$$  \label{partition_central_limit}
  \item[\textbf{C5}]
Condition Let $X \sim p_n$. \textbf{C3} holds, $\mathrm{var}(g_n(X)) = O(\sigma^2_n)$ and 
  $$\log{\mathbb{E}\left[e^{\beta \left(g_n(X)- \mathbb{E}g_n(X)\right)}\right]} \leq \frac{C\beta^2 \sigma_n^2}{1-|\beta| D\sigma_n}$$
  for all $\ |\beta| < \frac{1}{D\sigma_n}$  for absolute constants $C,D$ independent of $n$.

  \label{partition_concentration}
\end{enumerate}
\begin{remark}
In condition \textbf{C4}, we can relax the convergence to normal distribution
by considering convergence to a fixed distribution with a strictly positive tail. We have considered the standard normal distribution for the sake of clarity and since CLTs are ubiquitous and sufficient for the examples considered in this paper.
\end{remark}
\begin{remark}
We note that the function $g_n$ can have $\beta_n$ as a parameter but, condition \textbf{C3} requires that $g_n(x) - e_m(x)$ doesn't depend on $\beta_n$ whenever $x \in A_m$. Therefore, the conditional variances don't depend on the value of $\beta_n$. A trivial example is: $g_n(x) = l(x)+ \beta_n m(x)$. Other examples satisfying these conditions are given in Sections~\ref{sec:ising_model_result} and~\ref{sec:ergm_result}.  
\end{remark}

Define the function $m(x)$ such that $m(x) = m_0$ iff $x \in A_{m_0}$. We consider the following elementary binary hypothesis test between two distributions for the data $X \in \Omega_n$ and then extend this test to the composite case in Sections~\ref{sec:ising_model_result} and~\ref{sec:ergm_result}.
\begin{align*}H_0&: X \sim p_n
&\\ H_1&: X \sim q_n\quad \textrm{for some}\quad \beta_n\,.
\end{align*}
 The value of $\beta_n$ may be unknown.

  We call the following test to decide between $H_0$ and $H_1$ the canonical test with parameter $T \geq 0$ and a real-valued function $\kappa(\cdot)$ over the state space:
 \begin{definition}[Canonical Test]
 \label{def:canonical_test}
Given a sample $X$, we define the decision function $\mathcal{D}^{\mathsf{can}}(X) \in \{H_0,H_1\}$:
\begin{enumerate}
\item if $m(X) \notin S_n$ then $\mathcal{D}^{\mathsf{can}}(X) = H_1$
\item if $m(X) \in S_n$ and $\frac{\kappa(X)- e_{m(X)}(\kappa)}{\sigma_{m(X)}(\kappa)} \geq  T$ then $\mathcal{D}^{\mathsf{can}}(X) = H_1$
\item otherwise $\mathcal{D}^{\mathsf{can}}(X) = H_0$
\end{enumerate}
The statistical test with decision function $\mathcal{D}^{\mathsf{can}}$ is the canonical statistical test $\mathcal{T}^{\mathsf{can}}(T,\kappa)$.
\end{definition}
 We note that the canonical test depends only on the function $\kappa$, the set $S_n$ and the conditional measures $p^{(m)}$. A natural choice of $\kappa$ is: $\kappa = g_n$. We show the following result for this choice of $\kappa$. Our metric of comparison will the following `probability of error' for any test $\mathcal{T}$ with decision function $\mathcal{D}$:
$$p_{\mathsf{error}} = \max(\mathbb{P}[\mathcal{D}(X) = H_0| X\sim H_1], \mathbb{P}[\mathcal{D}(X) = H_1| X\sim H_0])\,.$$

\begin{theorem}
\label{main_theorem} Assume w.l.o.g that $\beta_n >0$. We have the following results.
\begin{enumerate}
\item
 If the conditions \textbf{C1},\textbf{C2},\textbf{C3}, and \textbf{C4}, hold then:
\begin{equation}
\lim_{n\to \infty }d_{\mathsf{TV}}(p_n,q_n) = 1 \quad \text{if  $\beta_n \sigma_n \to \infty$}
\end{equation}
Moreover, if it is known that $\beta_n\sigma_n \geq L_n$ for a known sequence $L_n \to \infty$ ($\beta_n$ being possibly unknown), then the canonical test $\mathcal{T}^{\mathsf{can}}(T_n,g_n)$ can distinguish between $p_n$ and $q_n$ with high probability with a single sample for a particular choice $T_n \to \infty$ depending only on $L_n $ and $\tau_n$. The probability of type 1 and type 2 errors can be bounded above by a function of $\alpha_n$, $T_n$ and $L_n$ tending to $0$.

\item
If condition \textbf{C5} holds, then
\begin{equation}
\lim_{n\to \infty }d_{\mathsf{TV}}(p_n,q_n) = 0 \quad \text{if  $\beta_n \sigma_n \to 0$}
\end{equation}
\end{enumerate}

\end{theorem}

We defer the proof to Appendix~\ref{sec: main_theorem_proof}. The idea behind the first part of the proof is described in Section~\ref{subsec: intuition}. To understand the proof of the second part of the theorem, we take $\Omega$ to be a finite space. Then, $q(x) = p(x)\frac{e^{\beta_n g(x)}}{\mathbb{E}_p e^{\beta_n g}}$. The Condition~\textbf{C5} along with Jensen's inequality implies that whenever $\beta_n\sigma_n \to 0$, $$e^{\beta_n \mathbb{E}_p g}\leq \mathbb{E}_p e^{\beta_n g} \leq e^{\beta_n \mathbb{E}_p g}e^{\frac{C\beta_n^2\sigma_n^2}{1-D|\beta_n|\sigma_n}} = (1+o(1))e^{\beta_n \mathbb{E}_p g}$$
Therefore, $q(x) = (1- o(1))p(x)e^{\beta_n(g(x) -\mathbb{E}_p(g))}$. We use Chebyshev inequality to show that $\beta_n (g(x) - \mathbb{E}_p(g))$ is small most of the time i.e,  $q(x) = (1\pm o(1))p(x)$ with high probability. This proves that the total variation distance converges to zero.

\section{Testing Ising Model Structure}
\label{sec:ising_model_result}
We intend to test between the following hypotheses for data $\hat{X} \in \{-1,1\}^n$:
\begin{enumerate}
\item
$H_0$ : The data is generated by a Curie-Weiss model at an unknown inverse temperature $0 \leq  \beta^{\mathrm{CW}}(n) \leq \beta_{\textrm{max}}$ and external field $|h^{\mathrm{CW}}| \leq h_{\textrm{max}} < \infty$
\item 
$H_1$ : The data is generated by an Ising model on a known d-regular graph at an unknown inverse temperature $0\leq\beta^{\mathrm{dreg}}_n < \infty$ and arbitrary external field $h^{\mathrm{dreg}} \in \mathbb{R}$ such that $(\beta^{\mathrm{dreg}}_n,h^{\mathrm{dreg}}) \in \Lambda_{\mathsf{Ising}}$
\end{enumerate}

We intend to apply Theorem~\ref{main_theorem} to prove Theorem~\ref{thm:ising_threshold}. For convenience, we use the notation used in the conditions for Theorem~\ref{main_theorem}. Let $x \in \Omega := \{-1,1\}^n$. We take $p_n$ to be Curie-Weiss model at inverse temperature $\beta^{\mathrm{CW}} \leq \beta_{\mathsf{max}}$ and external field $h^{\mathrm{CW}}$ such that $|h^{\mathrm{CW}}|\leq h_{\mathsf{max}} < \infty$ i.e, 
$$p_n(x) \propto e^{\frac{n}{2}\beta^{\mathrm{CW}} m^2 + nh^{\mathrm{CW}}m(x)}$$
Where $m := m(x) = \frac{1}{n}\sum_i x_i$. Let $G$ be any known $d$-regular graph over $n$ vertices with adjacency matrix $A$ and $d=o(n)$. We take $q$ to be the Ising model with interaction matrix $\beta^{\mathrm{dreg}}_n A$ and external field $h$ such that $\beta^{\mathrm{dreg}}_n > 0$ and $h^{\mathrm{dreg}} \in \mathbb{R}$. That is,
$$q_n(x) \propto e^{\frac{\beta^{\mathrm{dreg}}_n}{2}x^{\intercal}Ax + nh^{\mathrm{dreg}}m(x)}$$

We take $g_n(x) = \frac{1}{2}x^{\intercal}Ax - \frac{n}{2}\frac{\beta^{\mathrm{CW}}}{\beta^{\mathrm{dreg}}_n} m^2 + \frac{nd}{2(n-1)\beta^{\mathrm{dreg}}_n} + \frac{n(h^{\mathrm{dreg}} -h^{\mathrm{CW}})}{\beta^{\mathrm{dreg}}_n}m(x)$. Therefore, $$q_n(x) =  \frac{p_n(x)e^{\beta^{\mathrm{dreg}}_n g_n(x)}}{\mathbb{E}_{p_n} [e^{\beta^{\mathrm{dreg}}_n g(x)}]}$$
We take $M_n =\{-1,-1+\frac{2}{n},\dots, 1-\frac{2}{n},1\}$. Given $m_0 \in M_n$, define $A_{m_0} = \{x \in \Omega : m(x) = m_0\}$. Clearly, the subsets $A_{m_0}$ partition the set $\Omega$ and $A_{m_0} = \{x: |\{i: x_i = 1\}| = \frac{1+m_0}{2}n\}$. Magnetization concentration of Curie-Weiss model is well studied (c.f. \cite{ellis2007entropy}). The magnetization for the Curie Weiss model concentrates around the roots of the equation $m^{*} = \tanh{(\beta^{\mathrm{CW}}m^{*} + h^{\mathrm{CW}})}$. Since $\beta^{\mathrm{CW}} \leq \beta_{\mathsf{max}}< \infty$ and $|h^{\mathrm{CW}}| < h_{\mathsf{max}} < \infty$ we can show that for some $\epsilon >0$ and constants $B, C(\beta_{\mathsf{max}},h_{\mathsf{max}}) > 0$  depending only on $\beta_{\mathsf{max}}$ and $h_{\max}$,
 $$p_n\left(m(x) \in [-1+\epsilon, 1-\epsilon]\right) \geq 1 - Be^{-C(\beta_{\mathsf{max}},h_{\mathsf{max}})n} =: 1-\alpha_n \,.$$  
Therefore, we let $$S_n = M_n \cap [-1+\epsilon, 1-\epsilon]\,.$$

\begin{remark}
We immediately note that the following important fact: Consider the canonical test for $H_0$ and $H_1$ given in Definition~\ref{def:canonical_test}.
Given a sample $\hat{X}$ with magnetization $\hat{m} = m(\hat{X})$, we can determine whether $m(\hat{X}) \in S_n $ without using $(\beta^{\mathrm{dreg}}_n,h^{\mathrm{dreg}})$ and $(\beta^{\mathrm{CW}},h^{\mathrm{CW}})$ since $S_n$ only depends on $\beta_{\mathsf{max}}$ and $h_{\mathsf{max}}$. 
Clearly, $p^{(\hat{m})}$ is the uniform measure over $A_{\hat{m}}$ irrespective of the value of $\beta^{\mathrm{CW}}$ and $h^{\mathrm{CW}}$. Let $X_{\hat{m}} \sim p^{(\hat{m})}$
 A simple calculation shows that:
 $$\frac{1}{2}\hat{X}^{\intercal}A\hat{X}  - \mathbb{E}\left[\frac{1}{2}X_{\hat{m}}^{\intercal}AX_{\hat{m}}\right]= g(\hat{X}) - e_{\hat{m}}(g)$$
Therefore, $\sigma^2_m := \mathrm{var}(g(X_m)) =\mathrm{var}\left(\frac{1}{2}X_m^{\intercal}AX_m\right) \,.$ We observe that neither of the quantities above depend on the values of the unknown parameters and the decision whether $\frac{g(\hat{X}) - e_{\hat{m}}(g)}{\sigma_{\hat{m}}(g)} \geq T$ is the same irrespective of their value. We define $\kappa_{\mathsf{Ising}}(\hat{X}) := \hat{X}^{\intercal}A\hat{X}$. By the considerations above, we conclude that: $\mathcal{T}(T_n,g_n) = \mathcal{T}(T_n, \kappa_{\mathsf{Ising}})$. 
\label{rem:parameter_independence_ising}
\end{remark}

By Theorem~\ref{cut_size_clt}, $\sigma_m = \Theta(\sqrt{nd})$ uniformly for all $m \in S_n$ and $$\sup_{m \in S_n} d_{\mathsf{KS}}\left(\mathcal{L}\left(\tfrac{g(X_m) - e_m(g)}{\sigma_m}\right),\mathcal{N}(0,1)\right) <  C(\epsilon)\sqrt[4]{\frac{d}{n}} =: \tau_n$$

\begin{theorem} \label{thm:ising_threshold}
Let $d = o(n)$ and $L_n$ be any positive sequence diverging to infinity.
\begin{enumerate}
\item If $\Lambda_{\mathsf{Ising}} = \{(\beta^{\mathrm{dreg}}_n,h^{\mathrm{dreg}}) : \beta^{\mathrm{dreg}}_n \geq \frac{L_n}{\sqrt{nd}}, |h^{\mathrm{dreg}}| \leq h_{\mathsf{\max}}\}$, the canonical test $\mathcal{T}(T_n,\kappa_{\mathsf{Ising}})$,  which depends only on $\beta_{\mathsf{max}}$, $h_{\mathsf{max}}$ and $L_n$ can distinguish $H_0$ and $H_1$ with high probability for some choice of $T_n(\beta_{\mathsf{max}},h_{\mathsf{max}},L_n) \to \infty$.
\item If $\Lambda_{\mathsf{Ising}} = \{(\beta^{\mathrm{dreg}}_n,h^{\mathrm{dreg}}) : \beta^{\mathrm{dreg}}_n = \frac{1}{L_n\sqrt{nd}},|h^{\mathrm{dreg}}| \leq h_{\mathsf{\max}}\} $, there is no statistical test which can distinguish $H_0$ and $H_1 $ with high probability using constant number of i.i.d. samples.
\end{enumerate}
\end{theorem} 
 We defer the proof to Appendix~\ref{sec:ising_proof}. The idea is to use Remark~\ref{rem:parameter_independence_ising} to conclude $\mathcal{T}(T_n,g_n) = \mathcal{T}(T_n, \kappa_{\mathsf{Ising}})(\hat{X})$ and then use Theorem~\ref{main_theorem} to conclude the result.

We note from the proof that the above the threshold, the distribution $p_n$ need not necessarily be the Curie-Weiss model. It can be any family of permutation invariant probability distribution such that $p_n(m(x) \in [\delta,1-\delta]) \to 1$ for some $\delta > 0$ and our proof for the success of our statistical test goes through. But, below the threshold, our method cannot prove the total variation bound required if $p_n(\cdot)$ is not Curie-Weiss. In this sense, the Curie-Weiss model is the optimal adversary.

\section{A Central Limit Theorem for Quadratic Forms over Hamming Sphere}
\label{sec:clt} 
In order to apply Theorem~\ref{main_theorem} to problems of interest, we would like to prove a central limit theorem with Berry-Esseen type bounds for quadratic forms over Hamming Spheres. Consider $\mathcal{S} = \{(x_1,...,x_n) \in \{-1,1\}^n: |\{i: x_i = 1\}| = sn \}$. That is, $\mathcal{S}$ is the Hamming sphere of radius $sn$ for a fixed $s \in (0,1)$. Let $X \sim \mathrm{unif}(\mathcal{S})$. Given a symmetric matrix $A$ with $0$ diagonals, we intend to prove a central limit theorem for the quadratic form $\mathcal{A}(X) = \frac{1}{2}X^{\intercal}AX$. The problem of limiting distributions has been well studied for quadratic forms of i.i.d random variables (see \cite{hall1984central}, \cite{rotar1979limit}, \cite{de1987central}, \cite{gotze2002asymptotic}). All their methods utilize the independence of the entries of the random vector which is not case in this scenario. We use Stein's method to prove the following result: 

\begin{theorem}
\label{cut_size_clt}
Let $d = o(n)$ and $A$ be the adjacency matrix of a $d$ regular graph. Let $0<\delta<s<1-\delta<1$ and $\sigma^2_s := \mathrm{var}(\mathcal{A}(X)) $ and $L = \frac{\mathcal{A}(X)-\mathbb{E}\mathcal{A}(X)}{\sigma_s}$. Then,
\begin{enumerate}
\item
 $\sigma_s^2 = 8nds^2(1-s)^2(1+ O(\frac{d}{n}))$ 
 \item
 $d_{\mathsf{KS}}(\mathcal{L}(L),\mathcal{N}(0,1)) \leq C\sqrt[4]{\frac{d}{n}}$
\end{enumerate}
$C$ depends only on $\delta$ and the bound $O\left(\frac{d}{n}\right)$ holds uniformly for all $s \in (\delta,1-\delta)$.
\end{theorem}

A pair of random variables $(T,T^{\prime})$ is called exchangeable if $(T,T^{\prime}) \stackrel{d}{=} (T^{\prime},T)$.
\begin{definition}
We call a real valued exchangeable pair $(T,T^{\prime})$ an $a$-Stein pair with respect to the sigma algebra $\mathcal{F}$ if $T$ is $\mathcal{F}$ measurable and
$$\mathbb{E}(T^{\prime}|\mathcal{F}) = (1-a)T + a\mathbb{E}(T)$$
\end{definition}

We prove Theorem~\ref{cut_size_clt} using the following version of central limit theorem (Theorem 3.7 in \cite{ross2011fundamentals}).
\begin{theorem}
\label{stein_clt}
Let $(W,W^{\prime})$ be an $a$-Stein pair with respect to the sigma algebra $\mathcal{F}$ such that $W$ has 0 mean and unit variance. Let $N$ have the standard normal distribution. Then, 

$$d_{\mathsf{W}}(W,N) \leq \frac{\sqrt{\mathrm{var}\left(\mathbb{E}\left[(W^{\prime}-W)^2|\mathcal{F}\right]\right)}}{\sqrt{2\pi} a} + \frac{\mathbb{E}\left(|W-W^{\prime}|^3\right)}{3a}$$
\end{theorem}

We will find it convenient to think of the quadratic form $x^{\intercal}Ax$ in graph theoretic language. $A$ is the adjacency matrix of the $d$ regular simple graph $G$ - that is, $A_{i,j} \in \{0,1\}$ and $A_{i,j} = 1$ iff $(i,j) \in E(G)$. Consider the set $S(x) = \{i \in [n] : x_i = 1\}$. Let $\chi_S$ be the $n$ dimensional column vector such that $\chi_S(i) = 1$ if $i\in S$ and $\chi_S(i) = 0$ if $i \in S^{c}$. We shall henceforth use $S(x)$, $\chi_S(x)$ and $x$ interchangeably. Define $d(A,B)$ to be the number of edges of $G$ with one vertex in $A$ and the other in $B$. When $A= \{j\}$, we denote $d(A,B)$ be $d_{jB}$. We can easily show that 
\begin{equation}
\frac{1}{2} x^{\intercal}Ax = \frac{nd}{2} - 2d(S,S^c)
\label{eq:cut_size_identity}
\end{equation}
Therefore, it is sufficient to prove the CLT for $d(S(X),S(X)^c)$ when $S(X) \sim \mathrm{unif}(\mathcal{S})$. For the sake of clarity, we denote the random variable $S(X)$ by just $S$. Clearly, $|S| = sn =: l$. Define $T(S) := d(S,S^c)$. 

 We define the following exchangable pair $(S,S^{\prime})$ :
Draw $K$ and $J \in \{1,...,n\}$ uniformly at random and independent of each other and independent of $S$. Define $\chi_{S^{\prime}}$ to be the vector obtained by exchanging entries at indices $K$ and $J$ of $\chi_S$. 

Simple calculation using the fact that $G$ is $d$-regular, we can show that :
\begin{equation}
T(S^{\prime}) = \begin{cases}
T(S) &\quad \text{if $\chi_S (J) = \chi_S(K)$} \\
T(S) + 2(d_{J,S} - d_{K,S} + d_{J,K}) &\quad \text{if $J \in S$ and $K \in S^{c}$} \\
T(S) + 2(d_{K,S} - d_{J,S} + d_{J,K}) &\quad \text{if $K \in S$ and $J \in S^{c}$}
\end{cases}
\end{equation}

We apply Theorem~\ref{stein_clt} to the centered and normalized version of the Stein pair $(T(S), T^{\prime}(S))$ to prove Theorem~\ref{cut_size_clt}. We defer the proofs to Appendix~\ref{sec:clt_proof}.

\section{Concentration of Quadratic Forms over Hamming Sphere}
\label{sec:burkholder_subexp}
Let $S$ be the uniform random set of constant size and $T(S)$ be the size of the edge-cut, just like in Section~\ref{sec:clt}. Here, we relax the constraint on the size of $S$ so that $0 \leq |S| \leq n$. To lower bound the total variation distance, we need Condition~\textbf{C5}. To prove this condition, for the examples considered in this paper, we need sub-exponential bounds of the form:

\begin{equation}
\log{\mathbb{E}\exp{\left(\beta T - \beta \mathbb{E}T\right)}} \leq \frac{C\beta^2 nd}{1 - D\sqrt{nd}|\beta|}
\label{eq:subexp_wanted}
\end{equation}

We can easily show that the function $T(S)$ is $d$ Lipschitz in Hamming distance. Standard techniques give a sub-Gaussian bound of the form:

$$\log{\mathbb{E}\exp{\left(\beta T - \beta \mathbb{E}T\right)}} \leq C\beta^2 nd^2$$
The variance proxy of $nd^2$ is the equation above is much worse than the one in Equation~\eqref{eq:subexp_wanted}. This cannot give us the required sharp threshold when $d$ increases with $n$. In the case of centered independent random variables, i.e, when $y_i = \mathrm{Ber}(s)-s$, Hanson-Wright inequality for quadratic forms gives a sub-exponential concentration inequality like~\eqref{eq:subexp_wanted}. But it is not clear how to extend this to case when there are weak dependencies. 

To deal with this, tensorization of roughly the following form is normally proved:

$
\log{\mathbb{E}\exp{\left(\beta T - \beta \mathbb{E}T\right)}} 
\leq C \beta^2 \sum_{i=1}^n \mathbb{E}\Delta_i^2(T)
$.
Where $\Delta_i (f(x)) := f(x_i^{+}) - f(x_i^{-})$ is the discrete derivative. Here we run into a second problem: since our random set $S$ has constant size almost surely, we cannot remove a single element and the discrete derivative $\Delta_i f(x)$ cannot be defined within our space. We use the exchangeable pair used in Section~\ref{sec:clt} and Appendix~\ref{sec:clt_proof} to prove a well defined tensorization similar to the one above.

Using our method, based on Burkh\"older-Davis-Gundy type inequalities proved in \cite{chatterjee2007stein}, we show that:
\begin{equation}
\log{\mathbb{E}\exp{\gamma(T - \mathbb{E}T)}} \leq \frac{32nd\gamma^2(1+o(1))}{ 1 - 16nd\gamma^2(1+o(1))} 
\label{reverse_jensen}
\end{equation}
We defer the full proof to Appendix~\ref{sec: quadratic_form_concentration_proof}.

\section{Comparing ERGM to Erd\H{o}s-R\'enyi Model}

\label{sec:ergm_result}
Here, we compare $G(n,p_n)$ ($p_n \in (\delta,1-\delta)$ for some constant $\delta > 0$) to $\mathsf{ERGM}(\beta_1,\beta_2)$ which is the exponential random graph with $H_2$ being the  \hbox{(\shape{\Vshape})} graph.  Consider the following hypothesis testing problem given a single sample of a random simple graph $G$ over n vertices:

$H_0$ : $G$ is drawn from the distribution $G(n,p)$ for some $p \in (\delta,1-\delta)$

$H_1$ : $G$ is drawn from $\mathsf{ERGM}(\beta_1,\beta_2)$ for $\beta_1 \in \mathbb{R}$ and $\beta_2 \in \mathbb{R}^{+}$ for unknown $\beta_1$ and $\beta_2$ such that $(\beta_1,\beta_2) \in \Lambda_{\mathsf{ERGM}}$

Given a sample graph $X$, we let $V(X)$ be the number of wedge graphs \hbox{(\shape{\Vshape})} in $X$.
\begin{theorem}
\label{thm:ergm_threshold}
Let $L_n$ be any positive sequence diverging to infinity.
\begin{enumerate}
\item If $\Lambda_{\mathsf{ERGM}} = \{(\beta_1,\beta_2): \beta_2 \geq L_n\frac{1}{\sqrt{n}}, \beta_1 \in \mathbb{R}\}$ then the canonical statistical test $\mathcal{T}(T_n,V)$, which depends only on $\delta$ and $L_n$, can distinguish $H_0$ and $H_1$ with high probability for some choice of $T_n(\delta,L_n) \to \infty$.
\item If $\Lambda_{\mathsf{ERGM}} = \{(\beta_1,\beta_2): 0\leq \beta_2 =\frac{1}{L_n\sqrt{n}}, \beta_1 \in \mathbb{R}\}$, then there is no statistical test which can distinguish $H_0$ and $H_1$ with high probability using constant number of i.i.d. samples.
\end{enumerate}
\end{theorem}
We proceed in a way similar to Section~\ref{sec:ising_model_result} by proving each of the conditions \textbf(C1) - \textbf(C5). We defer the  proof to Appendix~\ref{sec:ergm_proof}.

\section*{Acknowledgment}
This work was supported in part by the grants ONR N00014-17-1-2147, DARPA W911NF-16-1-0551, and NSF CCF-1565516. 

\bibliographystyle{abbrvnat}
\bibliography{references}
\appendix

\section{Proof of Main Abstract Theorem~\ref{main_theorem}}
\label{sec: main_theorem_proof}

We first consider the case when $\sigma_n \beta_n \to \infty$.
Given a sample from $p_n$ or $q_n$, we prove that the statistical test $\mathcal{T}^{\mathsf{can}}(T_n,g_n)$ succeeds with high probability for some choice of $T_n$. Let $\mathcal{D}^{\mathsf{can}}$ be the decision function associated with the test $\mathcal{T}^{\mathsf{can}}(T_n,g_n)$.\\Consider the type 1 error rate:
\begin{align}
\mathbb{P}\left(\mathcal{D}^{\mathsf{can}}(X) = H_1 | X \sim H_0\right) &= p_n\left(m(X) \notin S_n\right) +\sum_{m\in S_n} p_n \left(\frac{g(X)- e_{m}(g)}{\sigma_{m}} \geq T \biggr \rvert m(X) = m\right)p_n\left(A_m\right) \nonumber \\
&\leq \alpha_n + \sum_{m \in S_n}\left[1-\Phi(T) + d_{\mathsf{KS}}\left(\mathcal{L}\left(\tfrac{g(X_m) - e_m(g)}{\sigma_m}\right),\mathcal{N}(0,1)\right) \right]p\left(A_m\right) \nonumber \\
&\leq \alpha_n + 1-\Phi(T) + \tau_n \label{type_1_error}
\end{align}
Now consider the type 2 error rate:
\begin{align}
\mathbb{P}\left(\mathcal{D}^{\mathsf{can}}(X) = H_0| X \sim H_1 \right) &= q_n\left(\frac{g(X)- e_{m(X)}(g)}{\sigma_{m(X)}} < T,m(X) \in S_n\right)\nonumber\\
&= \sum_{m \in S_n} q_n\left(\frac{g(X)- e_{m}(g)}{\sigma_{m}} < T\middle| m(X) = m\right) q_n\left(A_m\right)\nonumber\\
&= \sum_{m \in S_n} \tfrac{q_n\left(\frac{g(X)- e_{m}(g)}{\sigma_{m}} < T\middle| X\in A_m\right)}{q_n\left(\frac{g(X)- e_{m}(g)}{\sigma_{m}} < 2T\middle| X\in A_m\right)+q_n\left(\frac{g(X)- e_{m}(g)}{\sigma_{m}} \geq2 T\middle| X\in A_m\right)} q_n\left(A_m\right)\nonumber\\
&\leq  \sum_{m \in S_n} \tfrac{q_n\left(\frac{g(X)- e_{m}(g)}{\sigma_{m}} < T\middle| X\in A_m\right)}{q_n\left(\frac{g(X)- e_{m}(g)}{\sigma_{m}} \geq 2T\middle| X\in A_m\right)} q_n\left(A_m\right)\nonumber\\
&= \sum_{m \in S_n} \tfrac{\int_{g < e_m + T\sigma_m } e^{\beta_n g} dp^{(m)}}{\int_{g \geq e_m + 2T\sigma_m } e^{\beta_n g} dp^{(m)}} q_n\left(A_m\right)\nonumber \\
&\leq \sum_{m \in S_n} \tfrac{ e^{\left(\beta_n e_m + T\beta_n\sigma_m\right)} }{p^{(m)}\left(\{g \geq e_m + 2T\sigma_m\}\right)e^{\left(\beta_n e_m + 2T\beta_n\sigma_m\right)}} q\left(A_m\right)\nonumber \\
&\leq  \sum_{m \in S_n} \tfrac{e^{-T\beta_n\sigma_m }}{1-\Phi(2T) - \tau_n}q\left(A_m\right) \nonumber \\
&\leq \tfrac{e^{-cT\beta_n\sigma_n }}{1-\Phi(2T) - \tau_n} \label{type_2_error}
\end{align}
We use the fact that for positive $x$ and $y$, $\max(x,y) \leq x+y$, equation~\eqref{type_1_error} and~\eqref{type_2_error}, to conclude that for every $T >0$ such that $1-\Phi(2T) > \tau_n$ the error rate $p_{\mathsf{error}}$
\begin{equation}
p_{\mathsf{error}}  \leq \alpha_n + 1 - \Phi(T) + \tau_n + \frac{e^{-cT\beta_n\sigma_n }}{1-\Phi(2T) - \tau_n} \leq \alpha_n + 1 - \Phi(T) + \tau_n + \frac{e^{-cTL_n }}{1-\Phi(2T) - \tau_n}   \label{error_bound}
\end{equation}
For $n$ large enough, $\tau_n + e^{-cL_n} < \frac{1}{2}$. For such $n$, we can pick $T = T_n > 0$ such that

$$1-\Phi(2T_n) = \tau_n + e^{-cL_n}$$
Clearly, $T_n \to \infty$, therefore, $1-\Phi(T_n) \to 0$ and $$\frac{e^{-cT_nL_n }}{1-\Phi(2T_n) - \tau_n} = e^{-c(T_n-1)L_n } \to 0$$
Using the equations above in equation~\eqref{error_bound}, we conclude that:
$$p_{\mathsf{error}} \leq \alpha_n + 1 - \Phi(T_n) + \tau_n + e^{-c(T_n-1)L_n} \to 0$$

Therefore, the decision function $\mathcal{D}^{\mathsf{can}}(X)$ has a vanishing error rate for the choice of $T = T_n$ made above. Let $A^{\mathsf{can}} = \{x\in \Omega: \mathcal{D}^{\mathsf{can}}(x) = H_0\}$ 
\begin{align}
d_{\mathsf{TV}}(p_n,q_n) &= \sup_{A \in \mathcal{F}_n}p_n(A)-q_n(A)\nonumber \\
&\geq p_n(A^{\mathsf{can}}) - q_n(A^{\mathsf{can}})\nonumber \\
&=1- p_n\left((A^{\mathsf{can}})^{c}\right) - q_n(A^{\mathsf{can}})\nonumber \\
&= 1 - \mathbb{P}(\mathcal{D}^{\mathsf{can}}(X) = H_1| X \sim H_0) - \mathbb{P}(\mathcal{D}^{\mathsf{can}}(X) = H_0| X \sim H_1)\nonumber\\
&\geq 1 - 2\max(\mathbb{P}(\mathcal{D}^{\mathsf{can}}(X) = H_1| X \sim H_0),\mathbb{P}(\mathcal{D}^{\mathsf{can}}(X) = H_0| X \sim H_1))\nonumber\\
&= 1-2p_{\mathsf{error}} \label{eq:tv_error_bound}
\end{align}

Using Equation~\eqref{eq:tv_error_bound} we conclude that whenever $\sigma_n\beta_n \to \infty$,
$$d_{\mathsf{TV}}(p_n,q_n) \to 1$$

We now consider the case $\beta_n\sigma_n \to 0$.
Consider the set $A_{g_n} = \{x \in \Omega : \frac{e^{\beta_n g_n}}{\mathbb{E}_{p_n} e^{\beta_n g_n}} < 1\}$. It can be easily shown that $A_{g_n} \in \mathcal{F_n}$ and $d_{\mathsf{TV}}(p_n,q_n ) = p_n(A_{g_n}) - q_n(A_{g_n})$. Let $Z_{g_n} := \mathbb{E}_{p_n} e^{\beta_n g_n}$. Since, $\beta_n\sigma_n \to 0$, the following inequalities hold when $\beta_n\sigma_n$ is small enough and any $T >0$

\begin{align}
d_{\mathsf{TV}}(p_n,q_n ) &= p_n(A_{g_n}) - q_n(A_{g_n})\nonumber\\
&= \int \mathbbm{1}_{A_{g_n}}\left(1-\frac{e^{\beta_ng_n}}{\mathbb{E}_pe^{\beta_n g_n}}\right)dp_n \nonumber \\
&= \int \mathbbm{1}_{A_{g_n}} \left( 1- e^{-|\beta_n g_n - \log{Z_{g_n}}|}\right)dp_n \nonumber \\
&\leq \int \left( 1- e^{-|\beta_n g_n - \log{Z_{g_n}}|}\right) dp_n\nonumber\\
&\leq \int \left(1 - e^{-|\beta_n g_n - \beta_n\mathbb{E}_{p_n}[g_n]|}e^{-|\log{Z_{g_n}} - \beta_n \mathbb{E}_{p_n}[g]|}\right)dp_n\nonumber\\
&\leq \int \left( 1 - e^{- \frac{A\beta_n^2 \sigma_n^2}{1-B|\beta_n| \sigma_n}}e^{-|\beta_n(g_n - \mathbb{E}_{p_n}[g_n])|}\right)dp_n \nonumber\\
&\leq p_n\left(|g_n - \mathbb{E}_{p_n}[g_n]|\geq T\right) + 1-\exp{\left(- \frac{A\beta_n^2 \sigma_n^2}{1-B\beta_n \sigma_n}\right)}e^{-\beta_nT} \nonumber\\
&\leq \frac{\sigma_n^2}{T^2} + 1-\exp{\left(- \frac{A\beta_n^2 \sigma_n^2}{1-B\beta_n \sigma_n}\right)}e^{-\beta_nT}\label{tv_upper_bound}
\end{align}

Where we have used the Chebyshev bound in the last step and the subexponentiality of $g_n$. The coefficients $(A,B)$ are consistent with coefficients $(C,D)$ in condition \textbf{C5}. Let $\gamma_n \to 0$ be any positive sequence such that $\frac{\beta_n\sigma_n}{\gamma_n} \to 0$. Let $T = \frac{\gamma_n}{\beta_n}$. Using this choice of $T$ in Equation~\eqref{tv_upper_bound}, we conclude that:
$$d_{\mathsf{TV}}(p_n,q_n) \leq \frac{\beta_n^2\sigma_n^2}{\gamma_n^2} + 1- \exp{\left(- \frac{A\beta_n^2 \sigma_n^2}{1-B\beta_n \sigma_n}\right)}\exp{\left(-\gamma_n\right)} \to 0.$$

\section{Proof of Central Limit Theorem}

\label{sec:clt_proof}

\begin{lemma}
$(T(S),T(S^{\prime}))$ is a $\lambda$-Stein pair with respect to $\mathcal{F}(S)$, where $\lambda = 4\frac{n-1}{n^2}$. Further, $\mathbb{E}[T(S)] = \tfrac{l(n-l)d}{n-1} $
\label{lem:expectation_stein_pair}
\end{lemma}
\begin{proof}
Clearly,

\begin{align}
\mathbb{E}[T(S^{\prime})|S] &= T(S) + \tfrac{4}{n^2}\sum_{j \in S}\sum_{k\in S^c} d_{j,S} - d_{k,S} + d_{j,k} \\
&= T(S) + \frac{4}{n^2} \sum_{j\in S}(d - d_{j,S^c})(n-l) -  \sum_{k\in S^c} ld_{k,S} + \sum_{j \in S}{k\in S^c} d_{j,k}\\
&= \left(1-4\frac{n-1}{n^2}\right)T(S)  + 4\frac{l(n-l)d}{n^2}
\end{align}
Using the fact that $\mathbb{E}T(S) = \mathbb{E}T(S^{\prime})$, we conclude the result.
\end{proof}
We shall henceforth shorten $T(S^{\prime})$ to $T^{\prime}$ and define $\lambda := 4\frac{n-1}{n^2}$. We list some elementary results about various moments. 

\begin{lemma}
For a $d$-regular graph, when $n-d-2 > l > d + 2$, if $l = \theta(n)$
\begin{enumerate}
\item $\mathbb{E}\sum_{j\in S}\sum_{k\in S^c} d_{j,S}^2 = l(n-l)\left(d^2\frac{(l-1)(l-2)}{(n-1)(n-2)} + d\frac{(l-1)(n-l)}{(n-1)(n-2)}\right)$
\item $\mathbb{E}\sum_{j\in S}\sum_{k\in S^c} d_{k,S}^2  = l(n-l)\left(d^2\frac{(l)(l-1)}{(n-1)(n-2)} + d\frac{(l)(n-l -1)}{(n-1)(n-2)}\right)$
\item $\mathbb{E}\sum_{j\in S}\sum_{k\in S^c} d_{j,S}d_{k,S} = d^2l(n-l)\frac{(l)(l-1)}{(n-1)^2} - \mathrm{var}(T)$
\item $\mathbb{E}\sum_{j\in S}\sum_{k\in S^c} d_{k,S}d_{j,k} = O(nd^2)$
\item
$\mathbb{E}\sum_{j\in S}\sum_{k\in S^c} d_{j,S}d_{j,k} = O(nd^2)$
\item
$\mathbb{E}\sum_{j\in S}\sum_{k\in S^c} d_{j,k}^2 = O(nd)$
\end{enumerate}
\label{useful_moments}
\end{lemma}
\begin{proof}
\begin{enumerate}
\item
$$\mathbb{E}\sum_{j\in S}\sum_{k\in S^c} d_{j,S}^2 = \frac{l(n-l)}{n}\sum_{j=1}^{n}\mathbb{E}(d_{j,S}^2| j\in S)$$  
Denoting the neighborhood of $j$ by $N(j)$,

$$\mathbb{E}(d_{j,S}^2| j\in S) = \sum_{a,b \in N(j)} \mathbb{P}(a \in S, b\in S | j \in S)$$ 
A simple computation of the probability gives the result.
\item 
proof similar to the previous part.
\item 
$$\mathbb{E}\sum_{j\in S}\sum_{k\in S^c} d_{j,S}d_{k,S} = \mathbb{E}\left[\left(ld - T(S)\right)T(S)\right] $$
Using the fact that $\mathbb{E}\left[T(S)\right] = d\frac{l(n-l)}{n-1}$ we arrive at the result.
\item
The result follows from the fact that
$$\sum_{j\in S}\sum_{k\in S^c} d_{j,k}d_{k,S} =  \sum_{k\in S^c} d_{k,S}^2$$
\item 
 
$$\sum_{j\in S}\sum_{k\in S^c} d_{j,k}d_{j,S} = \sum_{j \in S} d_{j,S}d_{j,S^c} = O(nd^2)$$

\item 
We note that since $G$ is a $d$-regular graph, $d_{j,k}^2 = d_{j,k}$. Therefore,
$$\mathbb{E}\sum_{j\in S}\sum_{k\in S^c} d_{j,k}^2 = \mathbb{E}d(S,S^c)$$

\end{enumerate}
\end{proof}

\begin{lemma}

\begin{equation}
\mathrm{var}(T) = \frac{1}{2\lambda}\mathbb{E}\left[(T-T^{\prime})^2\right] \,.
\label{eq:stein_pair_variance}
\end{equation}
If $l = \theta(n)$, then, $$\mathrm{var}(T) = 2dn\frac{l^2(n-l)^2}{n^4} + O(d^2)$$

Denoting $l = sn$ and $s \in (0,1)$,  $$\sigma^2 := \mathrm{var}(T) = 2ds^2(1-s)^2( 1 + O(\frac{d}{n}))$$. The $O(\frac{d}{n})$ holds uniformly for all $s \in [\delta,1-\delta]$ when $0<\delta<1-\delta<1$ 
\label{std_dev}
\end{lemma}

\begin{proof}
Equation~\ref{eq:stein_pair_variance} follows from the fact that $(T,T^{\prime})$ forms a $\lambda$-Stein pair.
$$
\mathrm{var}(T) = \frac{n^2}{8(n-1)}\mathbb{E}\left[\mathbb{E}\left[(T-T^{\prime})^2|S\right] \right]= \frac{n^2}{8(n-1)}\left(\mathbb{E}\frac{8}{n^2}\sum_{j\in S}\sum_{k \in S^c} (d_{j,s}+d_{j,k}-d_{k,S})^2\right) $$
$$ = \frac{1}{n-1}\mathbb{E} \sum_{j\in S}\sum_{k \in S^c}(d_{j,s}^2 + d_{k,S}^2-2d_{j,S}d_{k,S} + d_{j,k}^2 + 2d_{j,k}d_{j,S} - 2d_{j,k}d_{k,S})$$

We use Lemma~\ref{useful_moments} to compute this expectation.
\end{proof}

\begin{lemma}
Let $\gamma(s) = \sum_{i=0}^{r}a_i s^i$ be any polynomial such that $ 0 \leq \gamma(s) \leq \alpha $ $\forall \ s \in [s_1,s_2]$ such that $s_1 < s_2$, then $|a_i| \leq  C\alpha$ for some constant $C$ depending only on $r, s_1$ and $s_2$
\label{comparison_bound}
\end{lemma}
\begin{proof}
Choose distinct $x_i \in [s_1,s_2]$ for $i \in \{0,1,..,r\}$. Let $\bm{a} = \left[a_0\ a_1\ ..\ a_r\right]^{\intercal} $ and $\bm{b} = \left[\gamma(x_0)\ \gamma(x_1)\ ...\ \gamma(x_{r})\right]^{\intercal}$. Consider the Vandermonde matrix with entries $V_{i,j} = x_i^{j}$ for $i,j \in \ \{0,1,...,r\}$. $V$ is invertible since $x_i$ are distinct and $V\bm{a} = \bm{b}$. Therefore, $\bm{a} = V^{-1}\bm{b}$. Therefore $\|a\|_{\infty} \leq \|V^{-1}\|_{\infty}\|b\|_{\infty}$. Since $\|b\|_{\infty} \leq \alpha$, we obtain the result by setting $C =  \|V^{-1}\|_{\infty}$.
\end{proof}

\begin{definition}[function type]
Let $R$ be a subset of vertices of a given graph $G$.
We define the following classification of functions $f(R)$
\begin{enumerate}
\item We call $f$ to be of \textbf{type 1} of index $r \in \mathbb{N}$ if $f(R) = \left(d_{j,R} - d_{k,R} + d_{j,k}\right)^r \mathbbm{1}_{j\in R}\mathbbm{1}_{k \in R^c}  $, 
\item We call $f$ to be of \textbf{type 2} of index $r  \in \mathbb{N}$ if $$f(R) = \left(d_{j_1,R} - d_{k_1,R} + d_{j_1,k_1}\right)^{r_1}\left(d_{j_2,R} - d_{k_2,R} + d_{j_2,k_2}\right)^{r_2} \mathbbm{1}_{j_1\in S}\mathbbm{1}_{k_1 \in R^c} \mathbbm{1}_{j_2\in S}\mathbbm{1}_{k_2 \in S^c}$$ such that $r_1, r_2 \in \mathbb{N}$ and  $r = r_1 + r_2$.
\end{enumerate}
\end{definition}

Since the coordinates of the random set $S$ are dependent (because $|S| = l$), it is hard to bound moments of functions of $S$. Therefore, we draw a random set $\tilde{S}$ such that each vertex is included independently with probability $p=\frac{l}{n}$. As we shall see, $S$ is locally similar to $\tilde{S}$ and hence we can use the known tools for bounding moments of functions of independent variables to bound the moments of $f(S)$.
\begin{lemma}
Let $f$ be a function of type 1 or type 2 with $G$ being a d-regular graph. Then, the following are true. Let $\tau$ by the `type' of the function $f$.
\begin{enumerate}
\item
$f(R) = \sum_{h = 0}^{r+2\tau} g_h(R)$ $\forall R \subset V$
\item
If each vertex is included in the set $\tilde{S}$ independently with probability $p = \frac{l}{n}$, then,
$$\mathbb{E}f\left(\tilde{S}\right)= \sum_{h=0}^{r+2\tau} a_h p^h$$
for some constants $a_h \in \mathbb{Z}$.
\item
If the set $S$ is chosen uniformly at random from all vertex subsets of size $l$, then
$\mathbb{E}f(S) = \sum_{h = 0}^{r+2\tau} a_h \prod_{i=0}^{h-1} \frac{l-i}{n-i}$
\end{enumerate}
Where $g_h(S)$ is a function of the form $\sum_{i\in I} (-1)^{\eta_i} \mathbbm{1}_{S_i \subset S }$. Where $\eta_i \in \{-1,+1\}$, $S_i \subset V$, $|S_i| = h$ and $I_h$ is any finite index set.
\label{poly_form}
\end{lemma}

\begin{proof}
\begin{enumerate}
\item
We use the following identities:
$$d_{j,S} = \sum_{i \in N(j)} \mathbbm{1}_{i \in S}\,.$$
$$\mathbbm{1}_{i \in S^c} = 1 - \mathbbm{1}_{i \in S}\,.$$
Expanding the power and noting that $\mathbbm{1}_{S_1 \subset S} \mathbbm{1}_{S_2 \subset S} = \mathbbm{1}_{S_1 \cup S_2 \subset S}$, we obtain the result.
\item
This follows trivially since $\mathbb{E} \mathbbm{1}_{S_i \subset T} = p^{|S_i|}$ and if $g_h(T)$ is of the form above, $a_h = \sum_{i\in I} (-1)^{\eta_i}$.
\item
This follows from the fact that $\mathbb{E}\mathbbm{1}_{S_i \subset S} = \frac{{{n-|S_i|}\choose{l-|S_i|}}}{{{n}\choose{l}}} = \prod_{i=0}^{h-1} \frac{l-i}{n-i}$, where $h = |S_i|$. If $g_h(T)$ is of the form above, $a_h = \sum_{i\in I_h} (-1)^{\eta_i}$
\end{enumerate}

\end{proof}

\begin{lemma}
If $f$ is of type 1 or 2 for a $d$ regular graph G over $n$ vertices with a fixed index $r$. Let $\tau$ by the `type' of the function.
\begin{enumerate}
\item
$\mathbb{E}f(\tilde{S}) = O\left(d^{\frac{r}{2}}\right)$
\item
$|\mathbb{E}f(\tilde{S}) - \mathbb{E}f\left(S\right)| = O\left(\frac{d^{\frac{r}{2}}}{n} \right)$ when $p = \frac{l}{n}$
\item
$\mathbb{E}f(S) \leq Cd^{\frac{r}{2}}\left(1 + O(\frac{1}{n})\right)$
\end{enumerate}
\label{poly_bound}
\end{lemma}

\begin{proof}

\begin{enumerate}
\item
Let $f$ be of type 1. Then,
\begin{align}
|\mathbb{E}f(\tilde{S})| &\leq \mathbb{E} |d_{j,\tilde{S}} - d_{k,\tilde{S}} + d_{j,k}|^r \nonumber \\
&\leq \left( 1 + 2\left(\mathbb{E} |d_{j,\tilde{S}} - \mathbb{E}d_{j,\tilde{S}} |^r\right)^{\frac{1}{r}} \right)^r \label{minkow_bound}
\end{align}
Where the inequalities above follow from Minkowski's inequality and the fact that $d_{j,\tilde{S}}$ and $d_{k,\tilde{S}}$ are identically distributed.

$d_{j,\tilde{S}}$ is a $1$ Lipschitz function of $\tilde{S}$ with respect to Hamming distance. 
We use MacDiarmid's inequality to conclude that $$\mathbb{P}(|d_{j,\tilde{S}} - \mathbb{E}d_{j,\tilde{S}}| > t ) \leq 2 \exp^{-\frac{2t^2}{d}} $$
From the above, we obtain the estimate: $$\mathbb{E}|d_{j,\tilde{S}} - \mathbb{E}d_{j,\tilde{S}} |^r \leq \int_{0}^{\infty}2r t^{r-1} e^{-\frac{2t^2}{d}} = \left(\frac{r\Gamma(\frac{r}{2})}{4}\right) d^{\frac{r}{2}} = O(d^{\frac{r}{2}})$$

Plugging it back into equation~\eqref{minkow_bound}, we obtain the result.

For any type 2 function $g$, we use Cauchy Schwarz inequality to note that:

\begin{align*}
|\mathbb{E}g(\tilde{S})| &\leq \mathbb{E} |d_{j_1,\tilde{S}} - d_{k_1,\tilde{S}} + d_{j_1,k_1}|^{r_1} |d_{j_2,\tilde{S}} - d_{k_2,\tilde{S}} + d_{j_2,k_2}|^{r_2} \\
&\leq \sqrt{\mathbb{E} |d_{j_1,\tilde{S}} - d_{k_1,\tilde{S}} + d_{j_1,k_1}|^{2r_1}} \sqrt{\mathbb{E} |d_{j_2,\tilde{S}} - d_{k_2,\tilde{S}} + d_{j_2,k_2}|^{2r_2}}\\
\end{align*}

And note that $\sqrt{\mathbb{E} |d_{j_i,\tilde{S}} - d_{k_i,\tilde{S}} + d_{j_i,k_i}|^{2r_i}} = O\left(d^{\frac{2r_i}{2}}\right)$ for $i = 1,2$, as shown above, to conclude the result.

\item
We use Lemma \ref{poly_form} to conclude that $\mathbb{E}f(\tilde{S})= \sum_{h=0}^{r+2\tau} a_h p^h = L(p)$ . Using the result in part 1, we conclude that for some absolute constant depending only on r, $L(p) \leq \alpha := Cd^{\frac{r}{2}}$ for every $p \in [0,1]$. We then invoke Lemma \ref{comparison_bound} to show that $|a_h| \leq C_1\alpha$ for all $h \in \{0,1,...,r+2\tau\}$ and that 
$$|\mathbb{E}f(\tilde{S}) - \mathbb{E}f(S)| \leq \sum_{h=0}^{r+2\tau} |a_h| \left|\left(\frac{l}{n}\right)^h - \prod_{i=0}^{h-1} \frac{l-i}{n-i}\right|$$ For a fixed $r$, $|\left(\frac{l}{n}\right)^h - \prod_{i=0}^{h-1} \frac{l-i}{n-i}| = O(\frac{1}{n})$ for every $l \leq n$. Therefore,  

$$|\mathbb{E}f(\tilde{S}) - \mathbb{E}f(S)| \leq \left(\frac{C_2}{n}\right)d^{\frac{r}{2}}$$

\item
This follows from parts 1 and 2.
\end{enumerate}
\end{proof}

Using the fact that the co-ordinates of the vector $\chi_S$ are weakly dependent, we prove the following bound on the expectation of type 1 and type 2 functions. This gives an explicit bound on the constant $C(r)$ for every $l$, which will be useful when proving concentration inequalities for $d(S,S^c)$.

\begin{lemma}
\label{poly_bound_alternate}
If $f$ is a function of type 1 or type 2 of index $r$ and $0\leq l\leq n $ then

$$\mathbb{E}|f(S)| \leq C(r)d^{\frac{r}{2}}$$  where $C(r)$ is a constant depending only on $r$.
\end{lemma}
\begin{proof}
It is suffient to prove this result for type 1 functions since this implies the result for type 2 functions through Cauchy-Schwarz inequality. Also, it is sufficient to prove this result when $r$ is even since an application of Jensen's inequality for the concave function $x^{\frac{r-1}{r}}$ implies the result for odd integers.
Assume $r$ is even and $f$ is a type 1 function defined by:
$$f(S) = (d_{j,S}-d_{k,S} + d_{j,k})^r \mathbbm{1}_{j \in S}\mathbbm{1}_{k \in S^c}$$
Define variable $y_i(S) := \mathbbm{1}_{i \in S} $. We note that,

\begin{align*}
f(S)&= \left(\sum_{i\in N(j)\setminus {k}} y_i - \sum_{i_1 \in N(k)\setminus{j}} y_{i_1}\right)^r \mathbbm{1}_{j \in S}\mathbbm{1}_{k \in S^c} \\
 &\leq \mathbb{E}\biggr \rvert\sum_{i\in N(j)\setminus {k}} y_i - \sum_{i_1 \in N(k)\setminus{j}} y_{i_1}\biggr \rvert^r \mathbbm{1}_{j \in S}\mathbbm{1}_{k \in S^c} \\
&\leq \mathbb{E}\biggr \rvert\sum_{i\in N(j)\setminus {k}} y_i - \sum_{i_1 \in N(k)\setminus{j}} y_{i_1}\biggr \rvert^r
\end{align*}

Define $$g_{jk}(S) = \sum_{i\in N(j)\setminus {k}} y_i - \sum_{i_1 \in N(k)\setminus{j}} y_{i_1}$$
$g_{jk}$ is a function of $(y_i)_{i\in D_{jk}}$ where $D_{j,k}= \left(N(j)\setminus\{k\}\right)\Delta \left(N(k)\setminus\{j\}\right)$ and $|D_{j,k}| := h \leq 2(d-1)$. $g_{jk}$ is $1$ Lipschitz in Hamming distance.

We follow the concentration inequalities as given in Section 4.2 of  \cite{chatterjee2005concentration}. 
We fix $j$ and $k$ such that $j \neq k$. $y_{\sim r} := (y_i : i \in D_{j,k}\setminus r)$. Let $\mu_i$ be the law of $y_i$. Define the dependency matrix $L = (a_{rs})$  $r,s \in D_{j,k}$ to be a matrix such that $$ d_{\mathrm{TV}}(\mu_{r}(.|y_{\sim r}),\mu_{r}(.|\hat{y}_{\sim r})) \leq \sum_{s \in D_{j,k}} a_{rs}\mathbbm{1}_{y_s \neq \hat{y}_s }$$

Let $h_1 =  d_H(y_{\sim r})$ and $h_2 = d_H(y_{\sim r})$. We consider two cases:

\begin{enumerate}
\item  $l > h$

 \begin{align*}
d_{\mathrm{TV}}(\mu_{r}(.|y_{\sim r}),\mu_{r}(.|\hat{y}_{\sim r})) &= |\mu_{r}(1|y_{\sim r}) - \mu_{r}(1|\hat{y}_{\sim r})| \\
&= \left| \frac{{{n-h}\choose{l-h_1-1}}}{{{n-h+1}\choose{l-h_1}}}  - \frac{{{n-h}\choose{l-h_2-1}}}{{{n-h+1}\choose{l-h_2}}}\right| \\
&= \left|\frac{h_1-h_2}{n-h+1}\right| \\
&\leq \sum_{s \in D_{j,k}\setminus\{r\}} \frac{1}{n-h+1} \mathbbm{1}_{y_s \neq \hat{y}_s }
\end{align*}

\item $l \leq h$ \\
This is similar to the previous case. It is clear that $d_{\mathsf{H}}(y_{\sim r}) \leq l$ a.s. 
Therefore, simple calculation shows that 
\begin{equation}
\mu_{r}(1|y_{\sim r}) = 
\begin{cases}
0 &\quad \text{if $h_1 = l$} \\
\frac{l-h_1}{n-h+1} &\quad \text{ if $h_1 < l$}
\end{cases}
\end{equation}
Proceeding similar to the previous case, we conclude the result.
\end{enumerate}

Therefore, we set $a_{rs} = \frac{1}{n-h+1}$ when $r\neq s$ and $a_{rr} = 0$. $A$ is a symmetric matrix. Therefore, $\|A\|_2 \leq \|A\|_1 = \frac{h-1}{n-h+1}$. 
Applying theorem 4.3 from \cite{chatterjee2005concentration}, we have 
\begin{equation}
\mathbb{P}(|g_{j,k} -\mathbb{E}(g_{j,k})| > t) \leq 2\exp{\left(-\left(\frac{1-\frac{h-1}{n-h+1}}{h}\right)t^2\right)}
\label{eq:weakly_dependent_concentration}
\end{equation}

Since $h \leq 2(d-1) = o(n)$, we conclude that $g_{j,k}$ is subgaussian with a variance proxy of $\frac{h}{2}(1+o(1))$. We also note that $\mathbb{E}(g_{j,k}) = 0$. We can bound the centralised moments of a sub-Gaussian random variable from Equation~\eqref{eq:weakly_dependent_concentration} as shown in~\cite{boucheron2013concentration} Theorem~2.1 : 

\begin{equation}
\mathbb{E}(g_{jk})^{2q} \leq 2(q!)[h(1+o(1))]^q  \leq 2(q!)[2d(1+o(1))]^q\,,
\label{moment_sub_gaussian_bound}
\end{equation}
where $q \in \mathbb{N}$ is arbitrary. Taking $r = 2q$ yields the result.
\end{proof}

Let $Y(S) : = \frac{T(S) - \mathbb{E}T(S)}{\sigma}$. We intend to apply Theorem~\ref{stein_clt} to the Stein pair $(Y,Y^{\prime})$ when $d = o(n)$. 

We first bound the term $\frac{\mathbb{E}\left(|Y-Y^{\prime}|^3\right)}{3\lambda}$ in the following lemma. 

\begin{lemma}
\label{stein_bound_1}
$\forall \  s \in \ (\delta,1-\delta)$ such that $0<\delta < \frac{1}{2}$, we have $$\frac{\mathbb{E}\left(|Y-Y^{\prime}|^3\right)}{3\lambda} = O\left(\sqrt{\frac{1}{n}}\right)$$ and the bound is uniform for all $s \in \ (\delta,1-\delta)$.
\end{lemma}

\begin{proof}
Using Lemma ~\ref{std_dev},

\begin{equation}
\frac{\mathbb{E}\left(|Y-Y^{\prime}|^3\right)}{3\lambda} = \frac{\mathbb{E}\left(|T-T^{\prime}|^3\right)}{3\lambda \sigma^{3}} = \frac{C\mathbb{E}\left(|T-T^{\prime}|^3\right)(1+O\left(\frac{d}{n}\right))}{\sqrt{n}d^{\frac{3}{2}}s^3(1-s)^3}
\label{eq:bound_y_1}
\end{equation}
Conditioning on $S$, 
\begin{align*}
\mathbb{E}\left(|T-T^{\prime}|^3\right) &= \mathbb{E} \frac{16}{n^2} \sum_{j \in S} \sum_{k \in S^c} |d_{j,S} - d_{k,S} + d_{j,k}|^3\\
& = \frac{16}{n^2} \sum_{j \in V} \sum_{k \in V} \mathbb{E}\left(|d_{j,S} - d_{k,S} + d_{j,k}|^3 \mathbbm{1}_{j \in S}\mathbbm{1}_{k\in S^c}\right) \\
&=  O\left(d^{\frac{3}{2}}\right)
\end{align*}
Where we get the last relation using Lemma ~\ref{poly_bound}. Substituting in Equation~\ref{eq:bound_y_1}, we conclude the result.
\end{proof}

We now bound the second term.
Since $|Y-Y^{\prime}| = \frac{|T - T^{\prime}|}{\sigma}$. Therefore, 
\begin{align}
 \frac{\sqrt{\mathrm{var}(\mathbb{E}((Y^{\prime}-Y)^2|S))}}{\sqrt{2\pi}\lambda} &= \frac{1}{\sqrt{2\pi}\lambda \sigma^2}\sqrt{\mathrm{var}(\mathbb{E}((T^{\prime}-T)^2|S))}   \nonumber \\
  &= \frac{1}{\sqrt{2\pi}\lambda  \sigma^2}\sqrt{\mathrm{var}(\mathbb{E}((T^{\prime}-T)^2|S))}  \nonumber \\
  &= \frac{1}{\sqrt{2\pi}\lambda  \sigma^2}\sqrt{\mathrm{var}\left(\frac{8}{n^2}\sum_{j\in S}\sum_{k \in S^c}(d_{j,S}+d_{j,k}-d_{k,S})^2\right)} \nonumber \\
 &= \sqrt{\frac{2}{\pi}}\frac{1}{ (n-1) \sigma^2}\sqrt{\mathrm{var}\left(\sum_{j\in V}\sum_{k \in V}(d_{j,S}+d_{j,k}-d_{k,S})^2\mathbbm{1}_{j\in S}\mathbbm{1}_{k \in S^c}\right)}  \label{second_quantity}
\end{align} 

For $j,k \in V$, define $h_{j,k} (R) := (d_{j,R}+d_{j,k}-d_{k,R})^2\mathbbm{1}_{j\in R}\mathbbm{1}_{k \in R^c}$. Clearly,

\begin{equation}
 \mathrm{var}\left(\sum_{j\in V}\sum_{k \in V}(d_{j,S}+d_{j,k}-d_{k,S})^2\mathbbm{1}_{j\in S}\mathbbm{1}_{k \in S^c}\right) = \sum_{j,k,j_1,k_1 \in V} \mathrm{cov}(h_{j,k}(S),h_{j_1,k_1}(S))
 \label{var_covar_1}
\end{equation}

Using Lemma ~\ref{poly_bound}, when $p = \frac{l}{n}$,
\begin{equation}
\mathrm{cov}(h_{j,k}(S),h_{j_1,k_1}(S)) = \mathrm{cov}(h_{j,k}(\tilde{S}),h_{j_1,k_1}(\tilde{S})) + O(\frac{d^2}{n})
\label{covar_compare}
\end{equation}

Using equations \ref{var_covar_1} and \ref{covar_compare} we conclude
\begin{align}
\mathrm{var}\left(\sum_{j\in V}\sum_{k \in V}(d_{j,S}+d_{j,k}-d_{k,S})^2\mathbbm{1}_{j\in S}\mathbbm{1}_{k \in S^c}\right) &= 
\mathrm{var}\left(\sum_{j\in V}\sum_{k \in V}(d_{j,\tilde{S}}+d_{j,k}-d_{k,\tilde{S}})^2\mathbbm{1}_{j\in \tilde{S}}\mathbbm{1}_{k \in \tilde{S}^c}\right) \nonumber \\ &+ O (n^3d^2) \label{var_compare}
\end{align}

\begin{lemma}
\label{stein_bound_2}
$$\mathrm{var}\left(\sum_{j\in V}\sum_{k \in V}(d_{j,\tilde{S}}+d_{j,k}-d_{k,\tilde{S}})^2\mathbbm{1}_{j\in \tilde{S}}\mathbbm{1}_{k \in \tilde{S}^c}\right)  = O(n^3d^3)$$ uniformly for all $p \in [0,1]$. Using equation ~\ref{var_compare}, we conclude that $\forall s \in [a,b]$ with $0<a<b<1$, 
$$\mathrm{var}\left(\sum_{j\in V}\sum_{k \in V}(d_{j,S}+d_{j,k}-d_{k,S})^2\mathbbm{1}_{j\in S}\mathbbm{1}_{k \in S^c}\right) = O(n^3d^3)$$ uniformly.
\end{lemma}

\begin{proof}
The elements of $\tilde{S}$ are drawn i.i.d with probability of inclusion $p$. Define $\epsilon_i = \mathbbm{1}_{i \in \tilde{S}}$. Then, $\epsilon_i \sim \mathrm{Ber}(p)$ i.i.d for $1 \leq i \leq n$. We use $\bm{\epsilon}$ and $\tilde{S}$ interchangeably.

$$F(\bm{\epsilon}) := F(\tilde{S}) = \sum_{j\in V}\sum_{k \in V} h_{j,k}(\tilde{S}) $$

Let $\tilde{S}_i := \tilde{S} \setminus\{i\}$ and $\Delta^{i}_{j,k}\left(\tilde{S}_i\right) := h_{j,k}\left(\tilde{S}_i\right) - h_{j,k}\left(\tilde{S}_i \cup \{i\}\right)$. Since entries of the vector $\bm{\epsilon}$ are independent, we use Efron-Stein method to tensorize the variance as follows:

$$\mathrm{var}(F(\bm{\epsilon})) \leq \sum_{i=1}^{n} \mathbb{E}\mathrm{var}_i \left(F(\bm{\epsilon})\right)$$
Where $\mathrm{var}_i \left(F(\bm{\epsilon})\right) = \mathrm{var} \left(F(\bm{\epsilon})|\epsilon_{\sim i}\right)$
Now, when $\bm{\epsilon}_{\sim i}$ is fixed, $F(\bm{\epsilon})$ can take two values. Therefore, 
\begin{equation}
\mathrm{var}_i(F(\bm{\epsilon})) = p(1-p)\left(F(\tilde{S}_i) - F(\tilde{S}_i \cup \{i\}\right)^2 = p(1-p)\left(\sum_{j,k}\Delta^{i}_{j,k}(\tilde{S}_i)\right)^2
\label{tensor_var_delta}
\end{equation}
By Cauchy-Schwarz inequality,
\begin{equation}
\sqrt{\mathbb{E}\left(\sum_{j,k}\Delta^{i}_{j,k}(\tilde{S}_i)\right)^2} \leq \sum_{j,k} \sqrt{\mathbb{E}\left(\Delta^{i}_{j,k}(\tilde{S}_i)\right)^2}
\label{triangle_delta}
\end{equation}

Clearly, $\Delta^{i}_{j,k}(\tilde{S}_i) \neq 0$ only if one of the following is true:
\begin{enumerate}
\item
$j=i$ and $k \neq i$
\item
$j \neq i$ and $k = i$
\item
$j \in N(i)$ and $k \not \in N(i)\cup \{i\}$
\item
$j \not \in N(i)\cup \{i\}$ and  $k \in N(i)$
\end{enumerate}
For case 1, considering sub cases $k \in N(i)$ and $k \not \in N(i)$, we conclude:
$$\Delta_{i,k}^{i}(\tilde{S}_i) = - \left(d_{i,\tilde{S}_i} -d_{k,\tilde{S}_i}\right)^2\mathbbm{1}_{k\in \tilde{S}_i^c}$$
$d_{i,\tilde{S}_i} \sim \mathrm{Bin}(p,d)$, $d_{k,\tilde{S}_i} \sim \mathrm{Bin}(p,d-1)$ if $k \in N(i)$ and $d_{k,\tilde{S}_i} \sim \mathrm{Bin}(p,d)$ if $k \not \in N(i)$. We use the same Minkowski inequality - McDiarmid concentration argument as in Lemma ~\ref{poly_bound} to conclude that when $j =i$ and $k \neq i$
\begin{equation}
\mathbb{E}\left(\Delta^{i}_{i,k}(\tilde{S}_i)\right)^2 = O(d^2) 
\label{case_1_tensor_bound}
\end{equation}
By a similar argument for case 2, when $j \neq i$ and $k =i$,

\begin{equation}
\mathbb{E}\left(\Delta^{i}_{j,i}(\tilde{S}_i)\right)^2 = O(d^2) 
\label{case_2_tensor_bound}
\end{equation}
We consider case 3. Let $j\in N(i)$ and $k \not \in N(i)\cup\{i\}$. 
Then, $$\Delta_{j,k}^{i}(\tilde{S}_i) = - \left(2(d_{j,\tilde{S}_i} - d_{k,\tilde{S}_i} + d_{j,k}) -1\right)\mathbbm{1}_{j\in \tilde{S}_i}\mathbbm{1}_{k \in \tilde{S}_i^c}$$
Clearly, $d_{j,S_i} \sim \mathrm{Bin}(p,d-1)$ and $d_{k,S_i} \sim \mathrm{Bin}(p,d)$. Using similar reasoning as case 1, we conclude that when $j\in N(i)$ and $k \not \in N(i)\cup\{i\}$ 

\begin{equation}
\mathbb{E}\left(\Delta^{i}_{j,k}(\tilde{S}_i)\right)^2 = O(d) 
\label{case_3_tensor_bound}
\end{equation}

We can repeat a similar argument for case 4 to conclude that when
$j \not \in N(i)\cup\{i\}$ and $k \in N(i)$, 

\begin{equation}
\mathbb{E}\left(\Delta^{i}_{j,k}(\tilde{S}_i)\right)^2 = O(d) 
\label{case_4_tensor_bound}
\end{equation}

All the $O()$ in the bounds above are uniform for $p\in [0,1]$.
There are at most $2n$ pairs ${j,k}$ which satisfy cases 1 or 2. There are at most $2nd$ pairs which satisfy cases 3 or 4.
Therefore, using equations~\eqref{triangle_delta}~\eqref{case_1_tensor_bound}~\eqref{case_2_tensor_bound}~\eqref{case_3_tensor_bound}~\eqref{case_4_tensor_bound} 

$$\sqrt{\mathbb{E}\left(\sum_{j,k}\Delta^{i}_{j,k}(\tilde{S}_i)\right)^2}
 = 2nO(d) + 2ndO(\sqrt{d}) = O(nd^{\frac{3}{2}})$$
Therefore, we conclude from equation~\eqref{tensor_var_delta} that for every $i \in V$
$$\mathbb{E}(\mathrm{var}_i(F(\bm{\epsilon}))) = O(n^2d^3)$$
By Efron-Stein method, we conclude that
$$\mathrm{var}(F(\bm{\epsilon})) = O(n^3d^3)$$
\end{proof}

We bound the second term in Theorem~\ref{stein_clt}
\begin{lemma}
Let $s\in (\delta,1-\delta)$ with $0<\delta<\frac{1}{2}$.
\begin{equation}
 \frac{\sqrt{\mathrm{var}(\mathbb{E}((Y^{\prime}-Y)^2|S))}}{\sqrt{2\pi}\lambda} = O\left(\sqrt{\frac{d}{n}}\right)
\end{equation}
The bound above holds uniformly for $s\in (\delta,1-\delta)$.
\label{lem:second_term_clt}
\end{lemma}

\begin{proof}
Using Lemma ~\ref{stein_bound_2} in equation~\eqref{second_quantity} and using the fact that $\sigma^2 = \Theta(nd)$ for all $s \in [\delta,1-\delta]$ uniformly, we conclude 

\begin{equation}
 \frac{\sqrt{\mathrm{var}(\mathbb{E}((Y^{\prime}-Y)^2|S))}}{\sqrt{2\pi}\lambda} = O\left(\sqrt{\frac{d}{n}}\right)
 \label{stein_quantity_2}
\end{equation}
\end{proof}

\begin{proof}[\textbf{Proof of Theorem~\ref{cut_size_clt}}]
We use Lemmas~\ref{lem:second_term_clt} and~\ref{stein_bound_1} along with Theorem~\ref{stein_clt} to show that
$$d_{\mathsf{W}}(\mathcal{L}(Y), \mathcal{N}(0,1)) \leq C\sqrt{\frac{d}{n}}$$
We conclude the bound for the Kolmogorov metric using the fact that when one of the arguments has the standard normal distribution,
$d_{\mathsf{KS}} \leq C \sqrt{d_{\mathsf{W}}}$
for some absolute constant $C$.
\end{proof}

\section{Proof of Concentration of Quadratic Forms}
\label{sec: quadratic_form_concentration_proof}
We continue here from the end of Section~\ref{sec:burkholder_subexp}. We refer to \cite{chatterjee2007stein} for details of the exchangeable pairs method for concentration inequalities and theorem 2.3 in \cite{boucheron2013concentration} for properties of sub-gamma distributions.

We begin with the Stein pair $(S,S^{\prime})$ defined in Section~\ref{sec:clt} with $|S| = l$ and $0 \leq l \leq n$. Following the notation in \cite{chatterjee2007stein},  we take $F(S,S^{\prime}) := T(S) - T(S^{\prime})$. Then, 
$$f(S) := \mathbb{E}\left[F(S,S^{\prime})|S\right] = \lambda (T - \mathbb{E}(T))$$ and 
\begin{align}
\Delta(S) &:= \frac{1}{2}\mathbb{E}\left[(f(S) -f(S^{\prime}))F(S,S^{\prime})|S\right]\nonumber \\ 
&= \frac{\lambda}{2}\mathbb{E}\left[(T-T^{\prime})^2|S\right]\nonumber \\
&= \frac{4\lambda}{n^2}\sum_{j\in V}\sum_{k\in V} (d_{j,S} - d_{k,S}+d_{j,k})^2 \mathbbm{1}_{j\in S}\mathbbm{1}_{k \in S^c} \nonumber \\
&:= \frac{4\lambda}{n^2} \sum_{j\in V}\sum_{k \in V} g_{j,k}^2(S)\mathbbm{1}_{j\in S}\mathbbm{1}_{k \in S^c}  \label{delta_definition}
\end{align}

From Theorem 1.5 in \cite{chatterjee2007stein},
\begin{align}
\mathbb{E}((f(S))^{2q}) &\leq (2q-1)^q \mathbb{E}(\Delta(S)^q) \nonumber \\
\implies\mathbb{E} (T-\mathbb{E}(T))^{2q} &\leq \left(\frac{2q-1}{\lambda^2}\right)^{q} \mathbb{E}(\Delta(S)^q)\nonumber \\
&= 4^q\left(\frac{2q-1}{\lambda}\right)^{q} \mathbb{E}\left(\frac{1}{n^2}\sum_{j\in V}\sum_{k \in V} g_{j,k}^2(S)\mathbbm{1}_{j\in S}\mathbbm{1}_{k \in S^c}\right)^q \nonumber \\
&\leq 4^q \left(\frac{2q-1}{\lambda}\right)^{q} \frac{1}{n^2} \sum_{j\in V}\sum_{k \in V}\mathbb{E}\left[g_{j,k}(S)^{2q}\right] \nonumber \\
&\leq 2.4^{q}.\left(\frac{2q-1}{\lambda}\right)^{q}q![2d(1+o(1))]^q \nonumber \\
&\leq 2.(2q)!.\left(\sqrt{4nd(1+o(1))}\right)^{2q} \label{sub_gamma_moment}
\end{align}
Where we used Jensen's inequality in the third step and Equation~\eqref{moment_sub_gaussian_bound} in the fourth step.

Following the proof of Theorem 2.3 in \cite{boucheron2013concentration}, we conclude that for every $\gamma$ such that $ 2|\gamma|\sqrt{4nd(1+o(1))} < 1 $, 
\begin{equation}
\log{\mathbb{E}\exp{\gamma(T - \mathbb{E}T)}} \leq \frac{32nd\gamma^2(1+o(1))}{ 1 - 16nd\gamma^2(1+o(1))}  \nonumber
\end{equation}
Which is the required result in Equation~\eqref{reverse_jensen}
This follows from a simple power series argument.
\section{Proof of Theorem~\ref{thm:ising_threshold}}

\label{sec:ising_proof}
We use the notation established in Section~\ref{sec:ising_model_result}.

\begin{proof}[\textbf{Proof of Theorem~\ref{thm:ising_threshold}}]
Consider the first case: $\beta^{\mathrm{dreg}}_n \sqrt{nd} = \Theta(\sigma_n \beta^{\mathrm{dreg}}_n) \geq  L_n \to \infty $. We first fix the parameters $\beta^{\mathrm{dreg}}_n$, $\beta^{\mathrm{CW}}$, $h^{\mathrm{dreg}}$ and $h^{\mathrm{CW}}$. $p_n$ and $q_n$ satisfy  Conditions~\textbf{C1}-\textbf{C4} of Theorem~\ref{main_theorem} as shown in Section~\ref{sec:ising_model_result}. 

We invoke  Theorem~\ref{main_theorem} to conclude that for some choice of $T_n$ depending only on $\tau_n$, $L_n$ and $S_n$, the canonical test $\mathcal{T}^{\mathsf{can}}(T_n,g_n)$ can distinguish between $p_n$ (with given parameters $\beta^{\mathrm{CW}}$, $h^{\mathrm{CW}}$, $\beta^{\mathrm{dreg}}_n$ and $h^{\mathrm{dreg}}$) and $q_n$ with a single sample with probability of error 
$$p_{\mathsf{error}} \leq f(L_n,\alpha_n, \tau_n) \to 0$$
We conclude from Remark~\ref{rem:parameter_independence_ising} that the canonical tests $\mathcal{T}(T_n,g_n)$ and $\mathcal{T}(T_n,\kappa_{\mathsf{Ising}})$ are the same. Therefore, the test $\mathcal{T}(T_n,\kappa_{\mathsf{Ising}})$ has the same success probability for the same choice of $T_n$. $\kappa_{\mathsf{Ising}}$ doesn't depend on the unknown parameters. The parameters $S_n$, $\tau_n$ and $\alpha_n$ depend only on $\beta_{\mathsf{max}}$ and $h_{\mathsf{max}}$. Therefore, given $L_n$, $T_n$ can be chosen without the knowledge of the unknown parameters. The probability of error tends to $0$ uniformly for every choice of the unknown parameters. Hence, we conclude that the canonical test $\mathcal{T}(T_n,\kappa_{\mathsf{Ising}})$ succeeds with high probability for any choice of the unknown parameters.

We now consider the second case: $\beta^{\mathrm{dreg}}_n \sqrt{nd} = \Theta(\sigma_n \beta^{\mathrm{dreg}}_n) \leq \frac{1}{ L_n} \to 0 $. It is sufficient to prove that for a specific sequence $(\beta^{\mathrm{CW}}_n, \beta^{\mathrm{dreg}}_n)$ and external fields $(h^{\mathrm{dreg}},h^{\mathrm{CW}})$, $$d_{\mathsf{TV}}(p_n,q_n) \to 0\,.$$
We take $\beta^{\mathrm{CW}} = \frac{nd\beta^{\mathrm{dreg}}_n}{n-1}$ and $h^{\mathrm{dreg}} = h^{\mathrm{CW}}$. A simple calculation using Lemma~\ref{lem:expectation_stein_pair} we show that $$e_m(g) = \mathbb{E}g(X_m) = 0\,.$$ Using Equation~\eqref{eq:cut_size_identity}, we conclude $$g(X_m) - e_m(g) = -2\left(T\left(S_m\right) - \mathbb{E}\left[T\left(S_m\right)\right]\right)$$
Where $S_m = S(X_m)$.
Clearly, for $n$ large enough, we have $4\beta^{\mathrm{dreg}}_n\sqrt{4nd(1+o(1))} <1$. We shall denote $T_m :\stackrel{d}{=} T(S_m)$
\begin{align}
\mathbb{E}_p\left[e^{\beta^{\mathrm{dreg}}_n g}\right] &= \sum_{m_0\in M_n} p\left(m(x) = m_0\right)\mathbb{E}\left[e^{\beta^{\mathrm{dreg}}_n g(X_m)}\right]\nonumber \\
&= \sum_{m_0\in M_n} p\left(m(x) = m_0\right)\mathbb{E}\left[e^{-2\beta^{\mathrm{dreg}}_n \left(T_m- \mathbb{E}\left[T_m\right]\right)}\right]\nonumber \\
&\leq e^{\frac{128nd\left(\beta^{\mathrm{dreg}}_n\right)^2(1+o(1))}{ 1 - 64nd\left(\beta^{\mathrm{dreg}}_n\right)^2(1+o(1))} } \nonumber \\
&\leq e^{\frac{128nd\left(\beta^{\mathrm{dreg}}_n\right)^2(1+o(1))}{ 1 - 8\beta^{\mathrm{dreg}}_n\sqrt{nd(1+o(1))}} }
\label{eq:sub_exp_cut}
\end{align}
Where we have used Equation~\eqref{reverse_jensen} in the second to last step.
To bound the variance $\mathrm{var}_p{g}$, we note that $\mathbb{E}_p(g)  = 0$ and $\mathbb{E}g(X_m) = 0$. Therefore,
\begin{equation}
\mathrm{var}_p{g} = \mathbb{E}_p g^2 = \sum_{m_0 \in M_n}p\left(m(x) = m_0\right) \mathbb{E}g^{2}(X_{m_0}) = \sum_{m_0 \in M_n}p\left(m(x) = m_0\right)\sigma^2_{m_0}
\label{eq:var_identity}
\end{equation}
Clearly, $\beta^{\mathrm{CW}} \to 0$. Therefore, for $n$ large enough, $\beta^{\mathrm{CW}} < 1$. We refer to the large deviations result for Curie-Weiss model given in \cite{ellis2007entropy} to conclude that $p(|m(x)| > m_{\mathsf{max}}(h_{\mathrm{max}})) \leq C_1e^{-nC}$ for some positive constants $C_1$ and $C$. 
Clearly, $$\sigma_{m_0} \leq Cn^2d^2$$ for every $m_0$. Invoking Theorem~\ref{cut_size_clt}, we conlude that whenever $m_0 \leq m_{\mathsf{max}}(h_{\mathrm{max}})$, $\sigma_{m_0} \leq D\sqrt{nd}$ for some constant $D$. Plugging these results in Equation~\eqref{eq:var_identity}, we conclude that
\begin{equation}
\mathrm{var}_p[g] \leq Dnd + C_1n^2d^2e^{-nC}  = O(nd)
\label{eq:var_upper_bound}
\end{equation}

Therefore, using Equations~\eqref{eq:var_upper_bound} and~\eqref{eq:sub_exp_cut} we conclude that for this particular choice of $\beta^{\mathrm{CW}}$ and $h^{\mathrm{CW}}$, $p$ and $g$ satisfy Condition \textbf{C5}. Therefore, invoking the second part of Theorem~\ref{main_theorem}, we conclude that $$d_{\mathsf{TV}}(p_n,q_n) \to 0\,.$$ which proves our result.
\end{proof}

\section{Proof of Theorem~\ref{thm:ergm_threshold}}
\label{sec:ergm_proof}
Consider Erd\H{o}s-R\'enyi model over $n$ vertices. We let $N = {{n}\choose{2}}$, the maximum number of edges. Consider $\mathcal{X} := \{0,1\}^N$. We index elements of $x \in \mathcal{X}$ by tuples $e = (i,j)$ such that $i,j \in [n]$ and $i < j$. We can represent any simple graph $G= (V,E)$ over $n$ vertices by an element $x(G) \in \mathcal{X}$ such that $x(G)_e = 1$ iff $e \in E$. Henceforth we use `$e$th component of $x$' and `edge $e$' interchangeably.

Consider an $N\times N$ symmetric matrix $H$ such that $H_{e,f} \in \{0,1\}$ and $H_{e,f} = 1$ iff $e$ and $f$ have a common vertex. Clearly, $H$ is the adjacency matrix of a $d=2(n-2)$ regular graph over $N$ vertices.  We partition $\mathcal{X}$ into Hamming spheres. 
For $m \in \{0,1,\dots,N\}=: M_n$, define $A_m \subset \mathcal{X}$ to be the set of simple graphs over $n$ vertices with exactly $m$ edges. Here we define the function $E(x):= m(x)$ to be the number of edges in the graph associated with $x \in \mathcal{X}$. Clearly, the number of $V$ graphs \hbox{(\shape{\Vshape})} which are subgraphs of the graph represented by $x$ is $$V(x)  = \frac{1}{2}\sum_{e,f} x_e x_f H_{ef}  = \frac{1}{2}x^{\intercal}Hx$$
 
Let $\mathbbm{1}$ be the all one vector.  A simple calculation using the fact that $H$ is a regular matrix shows that:
\begin{align}
(2x-\mathbbm{1})^{\intercal}H(2x-\mathbbm{1}) &= 4x^{\intercal}Hx - 4\mathbbm{1}^{\intercal}Hx + \mathbbm{1}^{\intercal}H\mathbbm{1} \nonumber\\
&= 8V(x) - 8(n-2)E(x) + 2(n-2)N\label{set_transform}
\end{align}

Clearly, $2x- \mathbbm{1} \in \{-1,1\}^N$ and $|\{ e : 2x_e-1 = 1 \}|= E(x)$. 

 Let $\mu$ be the probability measure associated with $G(n,p_n)$ ($p(\cdot)$ in the notation of Theorem~\ref{main_theorem}) such that $\delta<p_n<1-\delta$ for some constant $\delta > 0$. We shall drop the subscript of $p_n$ for the sake of clarity. Since $E$ is a binomial random variable, we can easily show using McDiarmid's inequality that: 

$$\mu\left(E(x) \in \left[\frac{\delta}{2}N,\left(1-\frac{\delta}{2}\right)N\right]\right) \geq 1 - 2e^{-c(\delta)N} =: 1- \alpha_n$$
For some constant $c(\delta) > 0$. Therefore, we let 
$$S_n = M_n \cap \left[\frac{\delta}{2}N,\left(1-\frac{\delta}{2}\right)N\right]$$
We let $g(x) = \left(n\left(\frac{\beta_1 -\frac{1}{2}\log{\frac{p}{1-p}}}{\beta_2}\right)E(x) + V(x)\right)$.

\begin{remark}
Let $\mu^{(m)}$ be the conditional distribution $\mu(\cdot| E(x) = m)$ ($p^{(m)}$ in Section~\ref{sec:abstract_result}). Similar to Remark~\ref{rem:parameter_independence_ising} about Ising models, we note that $\mu^{(m)} $ is the uniform distribution over the graphs with fixed number of edges $m$ irrespective of the value of $p$ (i.e, uniform distribution over the set $A_m$). Proceeding as in Remark~\ref{rem:parameter_independence_ising}, let $\hat{X}$ be the given sample and $\hat{m} = m(\hat{X})$. We can decide whether $\hat{m} \in S_n$ without the knowledge of the unknown parameters and since
$$g(\hat{X}) - e_{\hat{m}}(g) = V(\hat{X}) - e_{\hat{m}}V(x)$$ and

$$\mathrm{var}_{\hat{m}}(g) = \mathrm{var}_{\hat{m}}(V)$$ 
we can decide whether $\frac{g(\hat{X}) - e_{\hat{m}}(g)}{\sigma_{\hat{m}}(g)} \geq T$ without the knowledge of the unknown parameters. We conclude that $\mathcal{T}(T_n,g_n) = \mathcal{T}(T_n,V)$
\label{rem:parameter_independence_ergm}
\end{remark}

Let $X_m \sim \mu^{(m)}$. Therefore, whenever $x \in A_m$ for $m \in \left[\frac{\delta}{2}N,\left(1-\frac{\delta}{2}\right)N\right] $, $2X_m -1$ satisfies the hypothesis for Theorem~\ref{cut_size_clt}. Using Equation~\eqref{set_transform} and the fact that $E(X_m) = m$ is a constant a.s. we conclude that:  
$$\mathrm{var}\left(g(X_m)\right) =: \sigma_m^2 = \Theta(n^3)$$
and
$$d_{\mathsf{KS}}\left(\mathcal{L}\left(\tfrac{g(X_m) - \mathbb{E}g(X_m)}{\sigma_m}\right),\mathcal{N}(0,1)\right) \leq C\sqrt[4]{\frac{1}{n}} =: \tau_n$$
Where we have used the fact that degree $d = \Theta(n)$ and number of rows/columns is $N = \Theta(n^2)$. All the $\Theta(\cdot)$ and bounds hold uniformly for all $m \in S_n$. Let $\beta_n := \frac{2\beta_2}{n}$. We take $\nu(x) = \mu(x)\frac{e^{\beta_n g(x)}}{\mathbb{E}_{\mu}e^{\beta_n g}} =: \mathsf{ERGM}(\beta_1,\beta_2)$.

To prove Theorem~\ref{thm:ergm_threshold}, we will need the following Lemma, where we get very small variance of a quadratic function by picking the right coefficient.
\begin{lemma}
\label{lem:super_concentration}
Let $p \in [0,1]$ be arbitrary. Then there exists an absolute constant $c$ such that whenever $\frac{2\beta_2}{n} =: \beta_n < \frac{c(1-o(1))}{n^3}$ for some absolute constant $c$ then for some choice of $\beta_1$ as a function of $p$ and $\beta_2$,
the following hold:
\begin{enumerate}
\item $\mathrm{var}_\mu(g) = O(n^3)= O(N^{\frac{3}{2}})$
\item
$\log{\mathbb{E}_{\mu} \left[e^{\beta_n(g - \mathbb{E}_{\mu}g)}\right]} \leq  \frac{Cn^3\beta_n^2(1+o(1))}{ 1 - D\beta_n\sqrt{n^3(1+o(1))}} +  \frac{C_1 \beta_n^2 n^2(1+o(1))}{1-|\beta_n|D_1n(1+o(1))}$
\end{enumerate}
Where $C$, $C_1$, $D$ and $D_1$ are absolute constants
\end{lemma}
We defer the proof of this Lemma to Appendix~\ref{sec:super_concentration}.

We proceed with the proof of Theorem~\ref{thm:ergm_threshold}.

\begin{proof}[\textbf{Proof of Theorem~\ref{thm:ergm_threshold}}]
Since $\beta_n := \frac{2\beta_2}{n}$, it is sufficient to consider the regimes:
$\beta_n n^{3/2} \to \infty$ and $\beta_n n^{3/2} \to 0$. By Remark~\ref{rem:parameter_independence_ergm}, $\mathcal{T}(T_n,g_n) = \mathcal{T}(T_n,V)$, which doesn't depend on parameters $\beta_1$, $p$ or $\beta_2$. The proof of the first part is similar to that in Theorem~\ref{thm:ising_threshold} and it follows from the discussion above and Theorem~\ref{main_theorem}.

We now assume $\beta_2 \leq \frac{1}{L_n}\frac{1}{\sqrt{n}}$ and fix $p \in [\delta,1-\delta]$.  To prove the second part it is sufficient to show one distribution in $H_0$ is near to one distribution in $H_1$ in the total variation sense. We will $\beta_1$ as a function of $p$ and $\beta_2$ such that $d_{\mathsf{TV}}(\mu_n,\nu_n) \to 0$.

 Using the notation of Theorem~\ref{main_theorem}, we have $\sigma_n = \Theta(n^{3/2})$.
 By Lemma~\ref{lem:super_concentration}, we conclude that Condition \textbf{C5} for Theorem~\ref{main_theorem} holds for some choice of $\beta_1$ and hence $d_{\mathsf{TV}}(\mu_n,\nu_n) \to 0$
\end{proof}

\section{Proof of Super Concentration}
\label{sec:super_concentration}
\begin{lemma}
If $E \sim \mathrm{Bin}(N,p)$, then,
$$\mathbb{E}(E-Np)^{2q} \leq q!C^qN^{q}$$
For some absolute constant $C$ independent of $N$
\label{binomial_subgaussian_moment}
\end{lemma}

\begin{proof}
By McDiarmid's theorem, 
$$\mathbb{P}\left(|E-Np|>t\right) \leq 2e^{-\frac{2t^2}{N}}$$
We use the equivalence of moment inequalities and sub-gaussian concentrations (refer Theorem 2.1 in \cite{boucheron2013concentration}) to conclude the result. 
\end{proof}
Let $X$ be a random vector taking values in $\{0,1\}^N$ such that its coordinates are i.i.d. $\mathrm{Ber}(p)$. Consider the function $h(X) = E^2(X) -(2pN + 1 - 2p)E(X)$. As we shall see, this choice of coefficients is special since it corresponds to a very small variance.

 Obtain the random variable $X^{\prime}$ as follows: Choose n random index $I \sim \mathrm{unif}([N])$. $X_i = X_i^{\prime}$ whenever $i \neq I$ and $X^{\prime}_I \sim \mathrm{Ber}(p)$ independent of $X$. Clearly, $(X,X^{\prime})$ is an exchangeable pair.
\begin{lemma}
\begin{enumerate}
\item
$(h(X),h(X^{\prime}))$ is an $\eta$-Stein pair with respect to $\mathcal{F}(X)$ where $\eta = \frac{2}{N}$. 
\item
$\mathbb{E}h(X) = -p^2N(N-1)$
\end{enumerate}
\end{lemma}
\begin{proof}
We shorten $E(X)$ to $E$. Let $a :=  -(2pN + 1 - 2p)$
\begin{align*}
\mathbb{E}\left[h(X^{\prime}) - h(X) |X\right] &= p\left(1-\tfrac{E}{N}\right)\left( (E+1)^2 + a(E+1) - E^2-aE\right) \\&\quad + \tfrac{E}{N}(1-p)\left((E-1)^2 + a(E-1)-E^2 -aE\right) \\
&= -\tfrac{2}{N}E^2 + \tfrac{E}{N} \left[2pN -2p+1 - a\right] + p(	1+ a) \\
&= -\tfrac{2}{N}h(X) -2p^2(N-1)
\end{align*}

Using the definition of a Stein pair and the fact that $\mathbb{E}h(X) = \mathbb{E}h(X^{\prime})$, we conclude the result.
\end{proof}
We proceed in the same way as Section~\ref{sec:burkholder_subexp}.

$$F(X,X^{\prime}) := h(X) - h(X^{\prime})$$
$$f(X) := \mathbb{E}\left[F(X,X^{\prime})|X\right] = \eta (h(X) - \mathbb{E}h(X))$$
\begin{align}
\Delta(X) &:= \frac{1}{2}\mathbb{E}\left[(f(X) -f(X^{\prime}))F(X,X^{\prime})|X\right]\nonumber \\ 
&= \frac{\eta}{2}\mathbb{E}\left[\left(h(X)-h(X^{\prime})\right)^2|X\right]\nonumber \\
&= \frac{\eta}{2}\left[  p\left(1-\tfrac{E}{N}\right)(2E-2pN+2p)^2 + \tfrac{E}{N}(1-p)(2E-2pN+2p-2)^2\right] \nonumber\\
&\leq 2\eta \left[p(E-pN+p)^2 + (1-p)(E-pN+p-1)^2 \right] 
\label{tensorisation_independent}
\end{align}
From Theorem 1.5 in \cite{chatterjee2007stein}, for every $q \in \mathbb{N}$,
\begin{align}
\mathbb{E}((f(X))^{2q}) &\leq (2q-1)^q \mathbb{E}(\Delta(X)^q) \nonumber \\
\implies\mathbb{E} (h(X)-\mathbb{E}h(X))^{2q} &\leq \left(\frac{2q-1}{\eta^2}\right)^{q} \mathbb{E}(\Delta(X)^q)\nonumber \\
&= 2^q\left(\frac{2q-1}{\eta}\right)^{q} \mathbb{E}\left[p(E-pN+p)^2 + (1-p)(E-pN+p-1)^2\right]^q\nonumber \\
&\leq  2^q\left(\frac{2q-1}{\eta}\right)^{q} \mathbb{E}\left[p(E-pN+p)^{2q} + (1-p)(E-pN+p-1)^{2q}\right]\nonumber \\
&= 8^q\left(\frac{2q-1}{\eta}\right)^{q} \mathbb{E}\left[p\left(\tfrac{E-pN}{2}+\tfrac{p}{2}\right)^{2q} + (1-p)\left(\tfrac{E-pN}{2}+\tfrac{p-1}{2}\right)^{2q}\right]\nonumber \\
&\leq \frac{8^q}{2}\left(\frac{2q-1}{\eta}\right)^{q} \left(\mathbb{E} \left[(E-pN)^{2q}\right] + p^{2q+1} + (1-p)^{2q+1}\right)\nonumber \\
&\leq C^{2q} (2q)! N^{2q}
\end{align}
Where we have used Equation~\eqref{tensorisation_independent} in the second step, Jensen's inequality for the convex function $\phi(x) = |x|^q$ in the third step, Jensen's inequality again for the function $\phi(x) = |x|^{2q}$ and Lemma~\ref{binomial_subgaussian_moment} in the final step.

Following the proof of Theorem 2.3 in \cite{boucheron2013concentration}, we conclude that for every $\gamma$ such that $ |\gamma| C N < 1 $, 

\begin{equation}
\mathbb{E}e^{\gamma\left[h(X)-\mathbb{E}h(X)\right]} \leq 2 \frac{C^2 \gamma^2 N^2}{1-|\gamma|CN}
\label{super_concentration}
\end{equation}
Where $C$ is an absolute constant.

We use Equation~\eqref{reverse_jensen} along with Equation~\eqref{set_transform} to conclude that for every $ m  \in \{0,\dots,N\}$ and some absolute constants $C$ and $D$,
\begin{equation} 
 \log\mathbb{E}e^{\beta_n g(X_m) } \leq \beta_n \mathbb{E}g(X_m) + \frac{Cn^3\beta_n^2(1+o(1))}{ 1 - D\beta_n\sqrt{n^3(1+o(1))}}
 \label{eq:conditional_concentration}
 \end{equation}

\begin{proof}[Proof of Lemma~\ref{lem:super_concentration}]
The bound on variance follows from the bound on MGF shown in the second part of the theorem after an application of Theorem 2.3 in \cite{boucheron2013concentration}. Therefore, it is sufficient to show the bound on the MGF.

\begin{align*}
\mathbb{E} g(X_m) &=  \mathbb{E}[g(X)| E(X) = m] \\
&= \mathbb{E}V(X_m) + n\left(\frac{\beta_1 -\frac{1}{2}\log{\frac{p}{1-p}}}{\beta_2}\right)m \\
&= \frac{n-2}{N-1}m^2 + n\left(\frac{\beta_1 -\frac{1}{2}\log{\frac{p}{1-p}}}{\beta_2}\right)m - \frac{n-2}{N-1}m
\end{align*}

Therefore, we can choose $\beta_1 $ a function of $p$ and $\beta_2$ such that :
$$\mathbb{E}[g(X)| E(X)] = \frac{n-2}{N-1}h(X) =  \frac{n-2}{N-1} \left(E^2 - (2pN +1 - 2p)E\right)$$

Therefore, 
\begin{align*}
\mathbb{E}_{\mu} \left[e^{\beta_n(g - \mathbb{E}_{\mu}g)}\right] &= \mathbb{E}_{\mu}\left[ \mathbb{E}\left[e^{\beta_n(g(X_m) - \mathbb{E}_{\mu}g)}\Big \rvert E(X) = m  \right]\right] \\
&\leq e^{\frac{Cn^3\beta_n^2(1+o(1))}{ 1 - D\beta_n\sqrt{n^3(1+o(1))}}}
\mathbb{E}_{\mu}\left[e^{\beta_n\frac{n-2}{N-1} (h(X)- \mathbb{E}_{\mu}h)}\right] \\
&\leq e^{\frac{Cn^3\beta_n^2(1+o(1))}{ 1 - D\beta_n\sqrt{n^3(1+o(1))}}} e^{ 2 \frac{C^2 \beta_n^2 n^2(1+o(1))}{1-|\beta_n|Cn(1+o(1))}}
\end{align*}
Here, we have used Equation~\eqref{eq:conditional_concentration}  and the fact that for this particular choice of $\beta_1$, $\mathbb{E}_{\mu} h = \frac{n-2}{N-1}\mathbb{E}_{\mu} g$. In the third step we have used Equation~\eqref{super_concentration}.
\end{proof}

\end{document}